\documentclass[11pt,letterpaper,oneside,reqno]{amsart}
\usepackage{longtable}
\usepackage{amsmath,amsfonts,amsthm,amscd,cancel,hyperref,stmaryrd,subfig,todonotes,url}
\usepackage{amssymb}
\usepackage{algpseudocode}
\usepackage{anysize,placeins,enumitem}
\usepackage[toc,page]{appendix}
\usepackage{bbm}
\usepackage{bbold}
\usepackage{booktabs}
\usepackage{cancel}
\usepackage{caption}
\usepackage{cases}

\usepackage{color}
\usepackage{epsfig}
\usepackage{pgfplots}
\pgfplotsset{compat=1.18}
\usepackage{float,mathtools}
\usepackage{placeins}
\usepackage{graphicx}
\usepackage{latexsym}

\usepackage{longtable,tabularx,tabulary}
\usepackage{mathrsfs}
\usepackage{mathtools}
\mathtoolsset{showonlyrefs}
\usepackage[singlelinecheck=false]{caption}
\usepackage[final]{showlabels}
\usepackage{subcaption}
\usepackage{textcomp}
\usepackage{textgreek}

\usepackage{thmtools}

\usepackage{titletoc}
\usepackage{todonotes}
\usepackage{ulem}
\usepackage{soul}
\usepackage{verbatim}
\usepackage{xcolor}
\usepackage{siunitx}

\DeclarePairedDelimiter\floor{\lfloor}{\rfloor}

 \renewcommand{\O}{{\mathcal{O}}}

\newcommand{\R}{{\mathbb{R}}}
\newcommand{\C}{{\mathbb{C}}}

\newcommand{\Z}{{\mathbb{Z}}}

\renewcommand{\Re}{{\mathfrak{Re}}}

\makeatletter
\let\@wraptoccontribs\wraptoccontribs
\makeatother

\newcommand{\s}{\sigma}

\numberwithin{equation}{section}
\newtheorem{theorem}{Theorem}
 \newtheorem{corollary}{Corollary}[theorem]
 \newtheorem{lemma}[theorem]{Lemma}

  \theoremstyle{remark}
 \newtheorem*{remark}{Remark}

\newcommand*{\thmref}[1]{Theorem~\ref{#1}}

\newcommand*{\corref}[1]{Corollary~\ref{#1}}

\reversemarginpar 

\makeatletter  
\def\l@subsection{\@tocline{2}{0pt}{2.5pc}{6pc}{}} \makeatother

\begin{document}

\title[Exponential Sums and Approximate Functional Equations for Zeta]{Explicit Exponential Sum Estimates and Approximate Functional Equations for the Zeta Function}

\author[N. Dhiman]{Natasha Dhiman}
\address{Department of Mathematics and Computer Science\\
University of Lethbridge\\
4401 University Drive\\
Lethbridge, Alberta\\
T1K 3M4 Canada}
\email{natasha.dhiman@uleth.ca}

\author[H. Kadiri]{Habiba Kadiri}
\address{Department of Mathematics and Computer Science\\
University of Lethbridge\\
4401 University Drive\\
Lethbridge, Alberta\\
T1K 3M4 Canada}
\email{habiba.kadiri@uleth.ca}

\author[E. Quesada-Herrera]{Emily Quesada-Herrera}
\address{Department of Mathematics and Computer Science\\
University of Lethbridge\\
4401 University Drive\\
Lethbridge, Alberta\\
T1K 3M4 Canada}
\email{emily.quesadaherrera@uleth.ca}

\begin{abstract}
We show an explicit version of the Van der Corput truncated Poisson summation formula (B-process). By using refined explicit exponential sum estimates, this improves the error term of previous explicit results by Patel and Yang (2024) and Arias de Reyna (2024). As an application, we obtain fine explicit estimates for the error terms in approximate functional equations for the Riemann zeta function, improving on previous explicit results of Simoni\v{c} (2020) in certain ranges.
\end{abstract}

\subjclass[2010]{Primary 11L07, 11M06; Secondary 11Y35
}

\keywords{Van der Corput estimate, exponential sums, Poisson summation, approximate functional equation, Riemann zeta function, explicit results}

\maketitle

\tableofcontents

\section{Introduction}
The Riemann zeta function $\zeta(s)$ can be expressed in terms of its Dirichlet series 
\[\zeta(s) = \sum_{ n\ge 1} n^{-s}\] 
when $\sigma>1$, where $s=\sigma+it$ is a complex number. The \textit{functional equation} for $\zeta(s)$ implies that we can also express it as a convergent Dirichlet series when $\sigma <0$:  
\[
\zeta(s) = \chi(s)\sum_{1\le n} n^{s-1}, \quad \text{where }\ 
    \chi(s) = 2^s \pi^{s-1} \Gamma(1-s) \sin\left(\frac{\pi s}{2}\right).
\]
Understanding $\zeta(s)$ in the \textit{critical strip} $\{s\in \C: 0< \s<1\}$, where all non-trivial zeros of $\zeta (s)$ lie, is a more subtle problem in analytic number theory.
An \textit{approximate functional equation} (AFE) approximates $\zeta(s)$ in the critical
strip by finite Dirichlet sums. These approximate formulas are classical tools for bounding zeta inside the critical strip, for establishing zero-density estimates, and for Dirichlet divisor problems (see, for instance, \cite{Titchmarsh}). 
\smallskip

Hardy and Littlewood proved in \cite{HL1921, HL1923}
two approximations in the critical strip: 
\begin{equation}\label{eq-AFE1}
    \zeta(s) = \sum_{1\le n \le x} n^{-s} - \frac{x^{1-s}}{1-s} + \O\big(x^{-\sigma}\big),
\end{equation}
valid for $\sigma \ge \sigma_0 > 0$ and $|t| \le 2\pi x/C$ (where $C>1$); and, for $0 < \sigma < 1$ and $2\pi x y = |t|$, 
\begin{equation}\label{eq-AFE2}
    \zeta(s) = \sum_{1\le n \le x} n^{-s} + \chi(s) \sum_{1\le n \le y} n^{s-1}
    + \O\big( x^{-\sigma}\big) + \O\big( |t|^{1/2-\sigma} \, y^{\sigma-1}\big).
\end{equation}
We refer to
\eqref{eq-AFE1} as the AFE of the ``first kind" and to \eqref{eq-AFE2} as the AFE of the ``second kind".%
\footnote{We caution the reader that this terminology is non-standard.}
Note that, in \cite{HL1921}, a previous version of \eqref{eq-AFE2} had $(\log t)$-factors in the error terms. 
\smallskip

We prove explicit versions of \eqref{eq-AFE1} and \eqref{eq-AFE2} by means of
explicit B-estimates for oscillatory exponential sums of the shape
\begin{equation}\label{def-oscil-exp-sum}
\sum_{a<n\leq b} g(n) e^{2\pi i f(n)},
\end{equation}
where $f$ and $g$ are real-valued functions with $g$ positive, satisfying
monotonicity conditions up to their second derivatives. Our results impose no
restriction on the length of the interval $(a,b)$.
\medskip
\paragraph{\bf Notation:} \ \\
Throughout the article, we write \[f=\mathcal{O}^*(g)\ \text{ if }\ |f(x)|\leq g(x)\ \text{for all} \ x \text{ in the relevant domain}.\]
We also use Vinogradov notation $f\ll g$ and $f=\O(g)$ interchangeably. 
\medskip 

\subsection{An explicit van der Corput approximate Poisson summation formula:} \ \\
The classical Poisson summation formula states that, under reasonable decay and regularity conditions for a function $F:\R\to\C$, we have 
\begin{equation}\label{eq:poisson}
    \sum_{n\in\Z} F(n) = \sum_{\nu\in\Z} \widehat F(\nu), 
\end{equation}
where $\widehat F(y):=\int_\R F(x)e^{-2\pi i x y}\, dx$ is the Fourier transform of $F$. 
Meanwhile, the B-process of van der Corput transforms an exponential sum into a shorter one, indexed by the integers in the range of $f'$. The first step in the B-process is a truncated version of \eqref{eq:poisson} due to van der Corput (see, for instance, \cite[Section 8.3]{Iwaniec-Kowalski}). 
For a phase function $f(x)$ with a monotonic derivative, such an approximate Poisson formula transforms the sum \eqref{def-oscil-exp-sum} into a sum of truncated Fourier integrals over a much shorter range. 
\smallskip 

Our first main result (\thmref{thm-VDC}) is of the shape:
\begin{equation}\label{van-der-Corput}
\sum_{a<n\leq b} g(n) e^{2\pi i f(n)} =  \sum_{\nu =0}^{\lfloor f'(a) \rfloor} \int_{a}^{b} g(x) e^{2\pi i (f(x)-\nu x)} dx + T_0(a,b),
\end{equation}
where $T_0$ is an explicitly bounded error term.
Classically, when $f'(b) \to 0$ while $f'(a)$ is large, we can prove the asymptotic bound as in \cite[Lemma 4.7]{Titchmarsh} or \cite[Proposition 8.7]{Iwaniec-Kowalski}:
\begin{equation}
    \label{asymp-T0-Titchmarsh}
    T_0(a,b) \ll_g \log(f'(a)).
\end{equation}
In Theorem \ref{thm-VDC}, we prove two distinct explicit estimates for $T_0(a,b)$ depending on the size of $f'(a)$. 
We state it here for the case when $g(n)$ is constant:
\begin{corollary}[An explicit Poisson summation formula]
\label{cor-VDC}
Let $f(x) $  be a real function with continuous, positive and strictly decreasing derivative $f'(x)$ in $[a,b]$. If $a,b \in \mathbb{Z} + \frac{1}{2}$, then for
$\delta = 1 - (f'(a)-\lfloor f'(a) \rfloor)$, we have:
\begin{equation}\label{eq:cor-VDC-1}
\sum_{a<n\leq b}e^{2\pi i f(n)} 
=\sum_{\nu =0}^{\lfloor f'(a) \rfloor}\int_{a}^{b} e^{2\pi i (f(x)-\nu x)} dx 
+ \mathcal{O}^*\big(T_0(a,b)\big)
\end{equation}
where
\begin{equation}\label{eq:cor-VDC-1-T0}
\begin{aligned}
T_0(a,b)
=& \frac{1}{\pi}\Big(\log(1+f'(a))
+\log(1+\lfloor f'(a) \rfloor) +\gamma 
\Big. \bigg. \\
\bigg. \big. & -\frac1{2(1+\lfloor f'(a) \rfloor)} -\frac1{2(1+ f'(a) )} 
- \frac{\Gamma'}{\Gamma}(\delta)+\log 2 + \frac{1}{f'(a)} \Big) .
\end{aligned}
\end{equation}
If, in addition, we assume the function $f \in C^2(a,b)$ and $|f''|$ is positive and decreasing on $[a,b]$, then, we replace the definition of $T_0(a,b)$ with 
\begin{equation}\label{eq:cor-VDC-partII}
T_0(a,b)
= 
  \frac{\log 2}{2\pi} + 
  \frac{|f'(b)| \mathcal{B}_{f'}(b) + |f'(a)| \mathcal{B}_{f'}(a) }{2\pi}
  +\frac{1}{2\pi (f'(a))} 
  + \frac{|f''(a)|}{2\pi^2} \Big( \frac{\mathcal{E}_1(\delta,f'(a))}{ f'(a)} 
  + \mathcal{E}_2(\delta,f'(a)) \Big)  ,
 \end{equation}
where $\mathcal{B}_f'$, $\mathcal{E}_1$ and $\mathcal{E}_2$ are respectively defined in \eqref{def:B_f(Z+1/2)}, \eqref{auxilliary_error1} and \eqref{auxillary_error2} with $N=0$.
\end{corollary}%
Recently, Patel \cite[Lemma 2.26]{PatelPHD} and Patel with Yang \cite[Lemma 2.1]{PatelYang2024}, as well as Arias de Reyna \cite[Lemma 4]{dereyna2024approximationzetafunctiondirichlet}, establish similar bounds 
with an explicit leading factor of $\displaystyle \frac{3}{\pi} \log(f'(a)+1)$ (with secondary error terms of constant size). 
In comparison, \corref{cor-VDC} gives two approximations, depending on the magnitude of $f'(a)$, namely:
\begin{enumerate}
\item[Part I.\ ] $\displaystyle T_0(a,b) \approx \frac{ 2 }{\pi}\,\log(1+f'(a)) $ when $f'(a)$ is not too large,
\item[Part II.] $\displaystyle T_0(a,b) \approx\frac{\log 2 }{2\pi}$ when $f'(a)$ is large. 
\end{enumerate}
More precisely, we note that Part II gives the uniform bound below.\footnote{Part II gives the uniform bound:
\[
T_0(a,b) \ll \left(1+\frac{1}{f'(a)}\right) \left(1+\frac{|f''(a)|}{\delta^3}\right).
\] 
The factor of ${\delta^{-3}}$, instead of ${\delta^{-1}}$ as in Part 1 and other works, is a tradeoff we choose to remove the $\log(f'(a))$ terms. In our applications, we take $\delta=1/2$. 
}
We prove Theorem \ref{thm-VDC} directly for a generic weight function $g(x)$, rather than deriving the weighted case from the constant $g(n)=1$ case via partial summation (as done in \cite[Lemma 4.10]{Titchmarsh} or \cite[Lemma 5]{dereyna2024approximationzetafunctiondirichlet}). We observed that this discards significant oscillatory cancellation. Instead, we incorporated $g(x)$ directly into the stationary-phase integral analysis, obtaining sharper, fully explicit bounds.
\smallskip

Relative to previous work of \cite[Lemma 6-7]{Karatsuba-Korolev-2007}, \cite[Lemma 2.26]{PatelPHD}, \cite[Lemma 2.1]{PatelYang2024}, as well as \cite[Lemma 4-5]{dereyna2024approximationzetafunctiondirichlet}, \thmref{thm-VDC} and \corref{cor-VDC} improve the leading logarithmic dependence in both parts. In Part I, we reduce the constant in front of the $\log f'(a)$-term, from $\frac{3}{\pi}$ to $\frac{2}{\pi}$, and in Part II, we remove the $\log f'(a)$ dependence entirely, in both cases under derivative conditions typical of applications.\footnote{More precisely, as $-\frac{\Gamma'}{\Gamma}(\delta) = \frac{1}{\delta} - \gamma + O(\delta)$, we require $\delta \gg 1/\log(1+f'(a))$, a condition satisfied here.} 
The generic case, \thmref{thm-VDC}, is stated in Section \ref{Section-Theorem-Bestimate} and proved in Section \ref{Section-Proof-Theorem-Bestimate}.
\smallskip

Theorem \ref{thm-VDC} also refines the explicit Poisson summation formula of KAratsuba and Korolev  \cite[Theorem]{Karatsuba-Korolev-2007} and Patel and Yang \cite[Lemma 2.3]{PatelYang2024}. The latter is used to obtain explicit bounds for $|\zeta(1/2+it)|$. The following version, proved in Section \ref{section-proof-cor-Explicit_B_estimate}, reduces the coefficient of the $\log (f'(a)-f'(b))$ term:
\begin{corollary}[A van der Corput B-process] \label{Explicit_B_estimate}
    Let $a, b \in \mathbb{Z} + \frac{1}{2}$ with $a < b$. Let $f(x)$ be a real-valued, three times differentiable function on $[a, b]$ with continuous, positive, and strictly decreasing derivative $f'(x)$ with $0 < f'(b) < 1$. Let $\delta = 1 - (f'(a)-\lfloor f'(a) \rfloor)$. For each integer $\nu \in \left(0, \lfloor f'(a) \rfloor\right]$, let $x_\nu \in [a, b]$ be the unique point such that $f'(x_\nu) = \nu$. Furthermore, suppose that for all $x \in [a, b]$:$$\lambda_2 \le \vert{}f''(x)\vert{} \le h_2 \lambda_2 \quad \text{and} \quad \vert{}f'''(x)\vert{} \le h_3 \lambda_3.$$
Then
\begin{multline}
\bigg| \sum_{a < n \le b} e^{2\pi i f(n)} - \sum_{ 0 < \nu  \le \lfloor f'(a) \rfloor} \frac{e^{2\pi i (f(x_\nu) - \nu x_\nu - 1/ 8)}}{|f''(x_\nu)|^{1/2}} \bigg| 
\le \frac{2.686}{\sqrt{\lambda_2}} + \frac{2 \cdot 3^{2/3}}{\pi^{2/3}} h_2 h_3^{1/3} (b - a) \lambda_3^{1/3} \\
+ \frac{2}{\pi} \log(f'(a) - f'(b)) + 
\mathcal{D}_f(a,b),\end{multline}
with
\begin{equation}
\mathcal{D}_f(a,b) = \frac{|f'(a)|\mathcal{B}_{f'}(a) +  |f'(b)|\mathcal{B}_{f'}(b)}{2\pi}  + 
\frac{h_2\lambda_2 \mathcal{E}_1 (\delta, f'(a) )}{2\pi^2 f'(a)}  
+ \frac{h_2\lambda_2 \mathcal{E}_2 (\delta, f'(a) )}{2\pi^2}  
+ \frac{1}{2\pi f'(a)} + 1.251 + \frac{\log 2}{2\pi}. \end{equation}
\end{corollary}
For simplicity, we reiterate the uniform bound
\[\mathcal{D}_f(a,b) \ll \left(1+\frac{1}{f'(a)}\right) \left(1+\frac{h_2 \lambda_2}{\delta^3}\right). \]
Replacing \cite[Lemma 2.3]{PatelYang2024} with this Corollary could refine Patel and Yang's subconvexity bound for zeta on the critical line \cite[Theorem 1.1]{PatelYang2024}, as 
van der Corput is applied in their Lemma 3.4, which evaluates the contribution from sums of the shape $\sum_{ t^{7/17} \ll n\ll \sqrt{ t } } n^{-1/2-it}$. 
\smallskip 

As a natural application of Theorem \ref{thm-VDC}, we establish explicit approximate functional equations of both the first and second kind for the Riemann zeta function. In this context, the sum \eqref{def-oscil-exp-sum} takes 
\[
g(u) = u^{-\s} \ \text{and}\ f(u)=\frac{t}{2\pi }\log u. 
\]
For the approximate functional equation of the first kind, Part~I is the
appropriate estimate: with $a \asymp t$ we have
$f'(a) = \frac{t}{2\pi a} \approx \frac{1}{2\pi} < 1$, so that
$\lfloor f'(a) \rfloor = 0$ and the sum of Fourier integrals in
\eqref{van-der-Corput} reduces to a single term.
For the second kind we apply Part~II, since $a \asymp \sqrt{t}$ gives
$f'(a) = \frac{t}{2\pi a} \asymp \sqrt{t}$.
\subsection{An explicit approximate functional equation of first kind for $\zeta(s)$:}\ \\
We establish here an explicit version of \eqref{eq-AFE1}. 
\begin{corollary}\label{cor:all_t}
Let $s = \sigma + it$ with $\s \in (0,1]$ and $t \geq  t_0 \geq 14 $. We have 
\begin{equation}\label{eq:AF1cor}
\bigg|  \zeta(s) - \sum_{ 1\leq n \leq t } n^{-s} \bigg|  \leq c_0\,t^{-\s},
\end{equation} 
where $c_0$ is defined in \eqref{eq:def-c0} and is computable for chosen values of $t_0$. For instance, 
\begin{align}
& c_0 = 1.2552
\ \text{for} \ t_0=14.13472,    \\
&c_0=1.2127 \ \text{for} \  t_0=3\cdot10^{12}.
\end{align}
\end{corollary}
The numerical thresholds for $t_0$ correspond, respectively, to a value just below the ordinate of the first non-trivial zero of $\zeta(s)$ ($\gamma_1 = 14.134725\dots$), and the height up to which the Riemann Hypothesis has been verified \cite{Platt_2021}.
\smallskip 

Applying \thmref{thm-VDC} offers an alternative to the methods of Kadiri \cite[Corollary 1.3]{Kadiri2013} and Simoni\v{c} \cite[Corollary 2]{Simonic2020}, who obtained $c_0=2.1946$ and $c_0=1.755$, respectively, when $t_0=14.13472.$ In addition, it refines the explicit constant $1+29/14 \approx 3.0714$ announced by Arias de Reyna \cite[Theorem 6]{dereyna2024approximationzetafunctiondirichlet}, which was proven using \cite[Lemma 4]{dereyna2024approximationzetafunctiondirichlet}-- an explicit version of the van der Corput approximate Poisson summation of \cite[Lemma 4.7]{Titchmarsh}.
\smallskip 

Theorem \ref{thm-AFE1} establishes the approximate functional equation under the assumption that $x \in \mathbb{Z} + 1/2$. Our above corollary removes this condition at a minor cost to the explicit constant. Proofs of \thmref{thm-AFE1} and  \corref{cor:all_t} are given in Section \ref{section-proof-AFE1}. 

\medskip
\subsection{An explicit approximate functional equation of the second kind for $\zeta(s)$:}\ \\
The next two results give explicit versions of \eqref{eq-AFE2} with logarithmic factors, and then with absolute constants on a bounded range.
The logarithm is intrinsic to the Poisson Formula route. The gain is that the resulting constants are much smaller than those obtained by contour integral methods throughout the range of computational interest.
\smallskip 

The first is a simplified numerical version of the approximate functional equation established in \thmref{thm-AFE2} (see Section \ref{section-AFE2}):
\begin{corollary}[An explicit AFE 2 with log-factors] \label{cor-AFE2}
Let $s = \sigma + it$ with $1/2 \le \sigma \le 1$ and $|t| \ge t_0 \ge 2\pi$. Assume $x, y \in \mathbb{Z} + 1/2$\footnote{
It should be possible to remove, a posteriori, the condition that $x,y\in\Z+1/2$, while retaining constants that are very close to those in Table \ref{tab:epsilon0_maxima}, by carrying out an analysis similar to that of Section \ref{sec:all_t} to obtain Corollary \ref{cor:all_t}. This would require more work than this case, and we leave it to the interested reader. Otherwise, our proof methods can be relaxed to requiring, say, $\max(x,y)\in\Z+\frac{1}{2}$ and $(\min(x,y)-\floor{\min(x,y)})\in[\varepsilon,1-\varepsilon]$, with an extra error term of size $\ll\frac{1}{\varepsilon}+\frac{1}{(1-\varepsilon)^3}.$ } satisfy $2\pi xy = |t|$ with $x, y \ge h \ge 1.5$\footnote{Since $x, y$ are positive half-integers, with $x, y \ge 1$, then taking $h \ge 1.5$ is not an additional assumption.}. We have the approximate functional equation
\begin{equation}\label{corAFE2-explicit}
\zeta(s) = \sum_{1 \le n \le x} n^{-s} + \chi(s)\sum_{1 \le m \le y} m^{s-1} + \mathcal{E}(\sigma, t, x, y),
\end{equation}
where the error term $\mathcal{E}(\sigma, t, x, y)$ satisfies
\begin{equation}
|\mathcal{E}(\sigma, t, x, y)| \le 
\begin{cases}
\left(
\frac{(1+\delta_0)\log x}{\pi} + \epsilon_0
\right)
 \, x^{-\sigma} & \text{if } x < y, \\
 \left(
\frac{ \log y}{\pi} + \epsilon_0
\right)
\, \left(\frac{t}{2\pi}\right)^{1/2 - \sigma} y^{\sigma - 1}& \text{if } x \ge y,
\end{cases}
\end{equation}
with the constants $\epsilon_0=\epsilon_0(h, t_0)$ and $\delta_0=\delta_0(t_0)$ respectively defined by 
\begin{equation}\label{def:epsilon0-delta0}
\epsilon_0 = 
\max_{\frac{1}{2}\le \s \le 1}
\mathcal{E}_0(\sigma, h, t_0),\ \text{and}\ \delta_0= \max_{\frac{1}{2}\le \s \le 1} C_0(\sigma, t_0) -1.
\end{equation}
Here, $C_0$ is defined in \eqref{def-C0} and $\mathcal{E}_0$  in \eqref{def-E0-all-x-y}. 
Table~\ref{tab:epsilon0_maxima} displays values for $\epsilon_0$ and $\delta_0$, which hold uniformly for $1/2 \le \sigma \le 1$ and vary with $t_0$, and depending with the relative size of $x$ and $y$.
\begin{table}[H]
\centering
\caption{\corref{cor-AFE2}: Values for $\epsilon_0$ and $\delta_0$ as defined in \eqref{def:epsilon0-delta0}}
\label{tab:epsilon0_maxima}
\begin{tabular}{|c|c|c|c|c|}
\hline
$t_0$ & $\epsilon_0 \ (x < y)$ & $\epsilon_0 \ (x = y)$ & $\epsilon_0 \ (x > y)$ & $\delta_0(t_0)$ \\ \hline
$2\pi$ & $2.264445$ & $ 2.265204$ & $2.265204$ & $5.961915 \times 10^{-2}$ \\ \hline
$10^3$ & $1.792711$ & $1.265977$ & $1.792736$ & $3.692900 \times 10^{-4}$ \\ \hline
$10^{10}$ & $1.750701$ & $1.160079$ & $1.750701$ & $3.692588 \times 10^{-11}$ \\ \hline
$3 \times 10^{12}$ & $1.750689$ & $1.160048$ & $1.750689$ & $1.230863 \times 10^{-13}$ \\ \hline
\end{tabular}
\end{table}
Values of $\epsilon_0$ in Table \ref{tab:epsilon0_maxima} are calculated using $h=1.5$ when $x \neq y$, and $h=\lfloor \sqrt{\frac{t_0}{2\pi}}\rfloor +0.5$ when $x=y$.
\end{corollary}
The following formulation of \corref{cor-AFE2} gives an explicit version of \eqref{eq-AFE2} when $\min(x,y)$ is in a bounded range. 
\begin{corollary}[An explicit AFE 2 with absolute constants] \label{cor-k-AFE2}
Let $s = \sigma + it$ with $1/2 \le \sigma \le 1$ and $|t| \ge t_0 \ge 
2\pi
$. Assume $x, y \in \mathbb{Z} + 1/2$ satisfy $2\pi xy = |t|$ with $x, y \ge h \ge 1.5$. Let $k\in\Z_+$ be such that $1\le k \le 50$ and define $h_k=\floor{e^{k-1}}+0.5$ and $H_k=e^{k}$.
If $h_k \le \min(x,y) \le H_k$, then
\eqref{corAFE2-explicit} holds with 
\begin{equation}
|\mathcal{E}(\sigma, t, x, y)| \le 
\begin{cases}
\epsilon_k \, x^{-\sigma}& \text{if }\  x<y , \\
\epsilon_k \, \left(\frac{t}{2\pi}\right)^{1/2 - \sigma} y^{\sigma - 1} & \text{if }\ y \le x .
\end{cases}
\end{equation}
The constant $\epsilon_k=\epsilon_k( h_k,t_0)$ is defined by
\begin{equation}\label{def:epsilonk}
\epsilon_k = 
\begin{cases}
\frac{k(1+\delta_0)}{\pi} +\epsilon_0 & \text{if }\  x < y , \\
\frac{k}{\pi} + \epsilon_0 & \text{if }\  y \le x ,
\end{cases}
\end{equation}
with $\epsilon_0$ and $\delta_0$ as in \eqref{def:epsilon0-delta0}.
Tables \ref{tab:Epsilon_k} and \ref{tab:Epsilon_k-large} display values for the constant $\epsilon_k$.
\begin{table}[H] 
\centering
\caption{\corref{cor-k-AFE2}: Values for $\epsilon_k$ as defined in \eqref{def:epsilonk} with $t_0 = 10^{10}$}
\label{tab:Epsilon_k}
\begin{tabular}{|c|c|c|c|c|}
\hline
$k$ & $h$ & $H = e^k$ & $\epsilon_k\ (x < y)$ & $\epsilon_k\ (x \ge y)$ \\
\hline
$1$  & $1.50$    & $2.72$     & $2.069008$ & $2.069011$ \\
$2$  & $2.50$    & $7.39$     & $2.132265$ & $2.132269$ \\
$3$  & $7.50$    & $20.09$    & $2.222296$ & $2.222299$ \\
$4$  & $20.50$   & $54.60$    & $2.472235$ & $2.472238$ \\
$5$  & $54.50$   & $148.41$   & $2.766222$ & $2.766225$ \\
$6$  & $148.50$  & $403.43$   & $3.075276$ & $3.075279$ \\
$7$  & $403.50$  & $1096.63$  & $3.390198$ & $3.390201$ \\
$8$  & $1096.50$ & $2980.96$  & $3.707261$ & $3.707264$ \\
$9$  & $2980.50$ & $8103.08$  & $4.025112$ & $4.025115$ \\ 
$10$ & $8103.50$ & $22026.47$ & $4.343253$ & $4.343256$ \\
\hline
\end{tabular}
\end{table}
\end{corollary}
Note that the fluctuations in the $\epsilon_k$-values, for small values of $k$, are related to the balance between the linear growth of $k$ in these dyadic intervals (coming from the logarithmic term in \corref{cor-AFE2}, and the decreasing term $\epsilon_0$). The term $\epsilon_0(h_k)$ varies by less than $10^{-4}$ once $k\ge 10$. We also note that, for $t_0=10^{10}$ and $k\le 50$, we have $\frac{k\delta_0}{\pi}<10^{-8}.$ Therefore, for all $11\le k \le 50,$ we have, in all cases with $t_0= 10^{10}$, the bound
\begin{equation}\label{eq:ek-largeK}
    \epsilon_k \le \frac{k}{\pi} + 1.1601 .
\end{equation}
Table \ref{tab:Epsilon_k-large} gives upper bounds for $\epsilon_k$ in certain larger ranges of $k$.
\begin{table}[H] 
\centering
\caption{\corref{cor-k-AFE2}: Upper bounds for $\epsilon_k$ as defined in \eqref{def:epsilonk} with $t_0 = 10^{10}$.} 
\label{tab:Epsilon_k-large}
\begin{tabular}{|c|c|}
\hline
$k$ & $\epsilon_k\ $ \\
\hline
$11\le k \le 15$ &  $5.9346959$ \\
$16\le k \le 20$ &  $7.5262442$ \\
$21\le k \le 25$ &  $9.1177936$ \\
$26\le k \le 30$ &  $10.7093431$ \\
$31\le k \le 35$ &  $12.3008925$ \\
$36\le k \le 40$ &  $13.8924419$ \\
$41\le k \le 45$ &  $15.4839914$ \\
$46\le k \le 50$ &  $17.0755408$ \\
\hline
\end{tabular}
\end{table}
To further illustrate this, we highlight the following simple bound. Let $1/2 \le \sigma \le 1$, $|t| \ge 2\pi$, and  $x, y \in \mathbb{Z} + 1/2$ such that $2\pi xy = |t|$ with $x, y \ge 1$. Assume that $\min(x,y)\le 10^6$. Then,
\begin{equation}\label{eq:AFE2-simple}
|\mathcal{E}(\sigma, t, x, y)| \le 
\begin{cases}
6\, x^{-\sigma}& \text{if }\  x<y , \\
6 \, \left(\frac{t}{2\pi}\right)^{1/2 - \sigma} y^{\sigma - 1} & \text{if }\ y \le x .
\end{cases}
\end{equation}
\corref{cor-k-AFE2} allows for direct comparison with Simoni{\v{c}} \cite[Theorem 4, Theorem 6, Tables 2--4]{Simonic2020}. 
His argument in the non-symmetric case follows that of Hardy and Littlewood \cite[Theorem A]{HL1923}.
He also complements his work by calculating numerical bounds for the error terms in the Riemann--Siegel formula as proven by Arias de Reyna \cite[Theorems 4.1 and 4.2]{AriasDeReyna2011} in 2011.  
The error term $\mathcal{E}(\sigma,t,x,y)$ is denoted $R_1(s;x,y)$ in \cite{Simonic2020} with
\[
|R_1(s;x,y)| \le E x^{-\sigma} + F \left(\frac{|t|}{2\pi}\right)^{1/2-\sigma} y^{\s-1}. 
\]
We reproduce the values of $E+F$ in Table \ref{tab:EF_comp}. (We recall that the case $x=y$ is derived from \cite{AriasDeReyna2011}.)
\begin{table}[ht]
\centering
\caption{Values for $E+F$ from \cite[Tables 2--4]{Simonic2020} }
\label{tab:EF_comp}
\renewcommand{\arraystretch}{1.25}
\begin{tabular}{|c|c|c|c|}
\hline
$t_0$ & $x<y$ & $x=y$ & $x>y$ \\ 
\hline
$2\pi$    & $36.094$  & $4.257$   & $127.126$ \\
$10^3$    & $10.983$  & $1.195$   & $15.726$  \\
$10^{10}$ & $10.7502$ & $1.00007$ & $15.203$  \\ 
\hline
\end{tabular}
\end{table}

These values are directly comparable to our $\epsilon_k$ from Table \ref{tab:Epsilon_k} and \eqref{eq:ek-largeK}.
For instance, our estimates are finer than \cite{Simonic2020} in Table \ref{tab:EF_comp} in low ranges of $\min(x,y)$ where calculations are feasible:\\
$\epsilon_k < 15.203$ for $k\le 44$ , while $\epsilon_k < 10.7502$ for $k\le 30$.
Note that these points, $\min(x,y) \approx e^{44} \approx 10^{19.10}$ and $\min(x,y) \approx e^{30} \approx 10^{13.02}$, respectively, are the thresholds where Theorem \ref{thm-AFE2} provides finer explicit estimates. These are in the range in which $\zeta(s)$ has been computed, since our thresholds cover up to at least $t=2\pi xy \approx 10^{26}$.
In the symmetric case $x=y$, the constants of \cite{Simonic2020}, derived from
the Riemann--Siegel bounds of Arias de Reyna \cite[Theorems 4.1 and
4.2]{AriasDeReyna2011}, remain sharper than ours when $t_0\ge 10^3$. If $t_0=2\pi$, Corollary \ref{cor-AFE2} produces sharper bounds in the symmetric case, in the range $1\le x=y\le 3000.$
\smallskip 

Python code to verify the numerical constants in Table \ref{tab:epsilon0_maxima}, \ref{tab:Epsilon_k} and \ref{tab:Epsilon_k-large} from Corollary \ref{cor-AFE2} and \ref{cor-k-AFE2} can be found in the arXiv version of this paper (using floating-point arithmetic). 
The constants in the simple result stated in \eqref{eq:AFE2-simple} follow from Corollary \ref{cor-k-AFE2}, for $t_0=10^{10}$, by taking $k\le 15$ in Table \ref{tab:Epsilon_k} and \ref{tab:Epsilon_k-large}. This remains true for $t_0=2\pi$ by taking $k\le 14$, as can be verified with the code, and noting $10^6<e^{14}$.

\section{Preliminary Lemmas}\label{secn:prel-lemmas}

\subsection{Estimates of harmonic-like sums}
The following bounds are useful to establishing an explicit van der Corput $B$ estimate (Theorem \ref{thm-VDC}).
%
\begin{lemma} \label{lem:harmonic-lead}
Let $N$ be a positive integer and $0<y<N+1$, and $\Delta=N+1-y$. We have the following estimates:
   \begin{equation}
\label{bnd-sum-nu(nu-y)blower-bupper}
-  \frac1{y(N+1)} 
- \frac1y \frac{\Gamma'}{\Gamma}(\Delta) \leq \sum_{\nu>N}\frac{1}{\nu(\nu-y)} - \frac1y  \log(N+1)
\leq 
  - \frac1{2y(N+1)} - \frac1y\frac{\Gamma'}{\Gamma}(\Delta). 
\end{equation}
\begin{align}
\label{bnd-sum-nu(nu-y)^2bupper}
 \sum_{\nu>N} \frac{1}{\nu(\nu-y)^2} 
 \le& \frac{1}{y} \left( \frac{1}{\Delta^2} + \frac{1}{(\Delta+1)^2} + \frac{1}{\Delta+1} \right)  - \frac{1}{y^2} \left( \log(N+1) - \frac{1}{N+1} - \frac{\Gamma'}{\Gamma}(\Delta) \right).  \\
 \label{bnd-sum-nu(nu-y)^3bupper}
 \sum_{\nu>N} \frac{1}{\nu(\nu-y)^3} 
 \leq & \frac{1}{y} \left( \frac{1}{\Delta^3} + \frac{1}{(\Delta+1)^3} + \frac{1}{2(\Delta+1)^2} \right) 
 - \frac{1}{y^2} \left( \frac{1}{\Delta^2} + \frac{1}{\Delta+1} \right) 
 \nonumber \\
 & + \frac{1}{y^3} \left( \log(N+1) - \frac{1}{2(N+1)} - \frac{\Gamma'}{\Gamma}(\Delta) \right). \\
 \sum_{\nu\ge1} \frac{1}{\nu(\nu+y)^k} \leq & \frac{\log(y+1) + \gamma}{y^k} - \frac{q_k(y)}{2y^k(y+1)^{k-1}} \label{eq:sum_infty_power_k} ,\ \text{for}\ k=1,2,3,
\end{align}
with
$q_1(y)=(y+1)^{-1}$, $q_2(y)=1+2y$, $q_3(y)=1+3y+3y^2$.
\end{lemma}

\begin{proof}
We establish the above bounds via partial fraction decomposition, properties of the digamma function $\frac{\Gamma'}{\Gamma}$, and integral comparison. 
Details can be found in Appendix \ref{secn:proofs-lemmas} of the ArXiv version of this paper. 
\end{proof}
The following two sections provide estimates for exponential sums (Lemmas \ref{lem:geometric} and \ref{lemma_bnd_S_plus_minus}) crucial for establishing an explicit van der Corput $B$ estimate (Theorem \ref{thm-VDC}).
\subsection{Estimates of finite exponential sums}
We first need some auxiliary bounds on a geometric exponential sum and a twisted harmonic sum, which will appear in the proof of the next lemmas. 
\smallskip

For $x,y\in\R$, and specifically for $x$ not being an integer, we consider the finite sums 
    \begin{equation}
    \label{def-Sxy}
        S_0(x,y) = \sum_{1\leq \nu\leq y} {e^{-2\pi i \nu x}} \ \text{ and }\  
        S_1(x,y) = \sum_{1\leq \nu\leq y} \frac{e^{-2\pi i \nu x}}{\nu}.
\end{equation}
%
For $y<1$, we have the empty sums $S_0(x,y)=S_1(x,y)=0.$ Furthermore, we note that Titchmarsh (\cite{Titchmarsh}) and Arias de Reyna (\cite{dereyna2024approximationzetafunctiondirichlet}) use the bound $S_1(x,y) \leq  \log (1+|y|)$. 

\smallskip
This is essentially tight when $x\in\Z$. For non-integer $x$, we will prove some stronger bounds that take advantage of cancellation in the oscillations, which are our main source of improvement from previous results. A generic bound for $Z$ follows from one for $S$ by partial summation, and we can obtain a tighter bound when $x$ is a half-integer. 

\begin{figure}[ht]
	\centering
		\centering
		\includegraphics[width=3in]{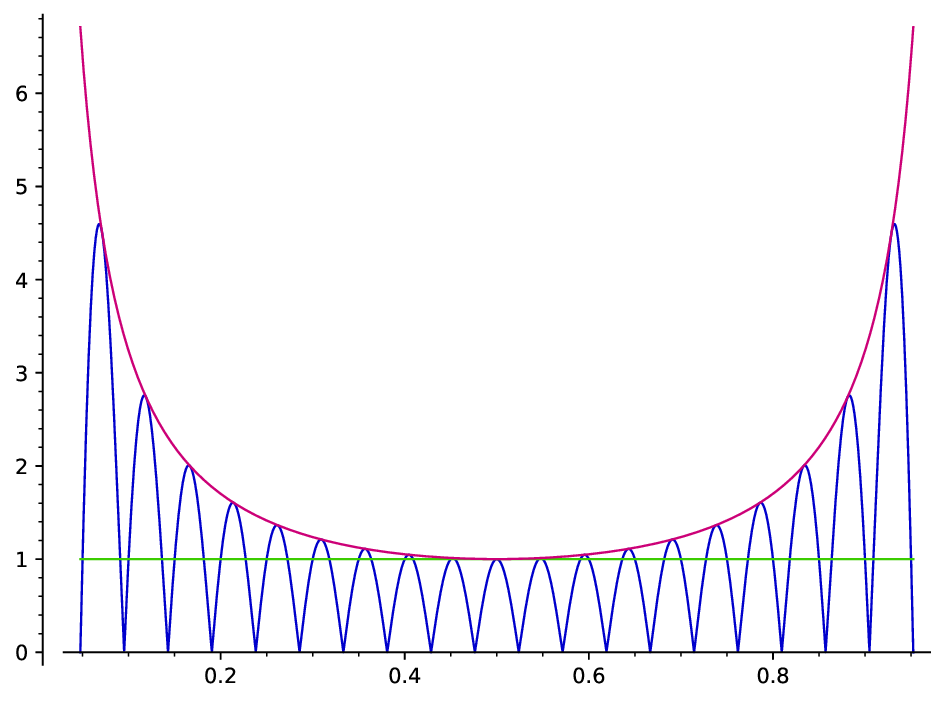}
		\caption{\label{fig:sxy}}Plot of $|S_0(x,y)|$ as a function of $x$, in the case $\lfloor y \rfloor = 21,$ in blue. In purple, the bound $1/|\sin(\pi x)|$.
		  
\end{figure}
\begin{lemma}\label{lem:geometric}
    For $x,y\in \R$ with $x\notin \Z$ and $y>0$, we have the bounds:
    \begin{align}
        |S_0(x,y)| & = \bigg| \frac{\sin(\pi x \lfloor y \rfloor )}{\sin \pi x} \bigg| \leq \begin{cases}
        \frac{1}{|\sin \pi x|} & \text{ if } y \geq 1,\\
        0 & \text{ if } 0<y<1,\\\end{cases}
        \label{bnd-S_0(y)}\\
        |S_1(x,y)|&\leq \tilde{S_1}(x,y)
       = \begin{cases}
        \frac{1}{|\sin(\pi x)|}\left(
        \frac{1}{y} +1 
        \right) & \text{ if } y \geq 1,\\
        0 & \text{ if } 0<y<1.\\\end{cases}
         \label{bnd-Zxy}
    \end{align}
    Moreover, in the special case where $x=n+1/2\  (n\in\Z)$ and $y\ge1$, we can replace the definition of $\tilde{S_1}$ with the slightly stronger bound:
   \begin{equation}\label{bnd-ZxySpecial}
        \tilde{S_1}(n+1/2,y)=\log 2 + \mathcal{O}^*\left(
        \frac{1}{y}
        \right).
   \end{equation}
\end{lemma}
\begin{proof}
We establish the above bounds via geometric summation for $S_0$ and integrating by parts with the resulting bound for $S_1$, via a Riemann-Stieltjes integral.
Details can be found in Appendix \ref{secn:proofs-lemmas} of the ArXiv version of this paper. 
\end{proof}
\begin{remark}
The finer behavior of $S_0(x,y)$ (and $S_1(x,y)$) depend strongly on the shape of $x$, particularly on how close it is to fractions of denominator $\lfloor y \rfloor$, as well as the parity of $\lfloor y \rfloor$. See Figure \ref{fig:sxy} for a plot of $|S_0(x,y)|$ as a function of $x$, when $\lfloor y \rfloor = 21.$ 
It might be possible to decrease the factor $(1/y+1)$ in the generic bound \eqref{bnd-Zxy} to $(1/y+\log 2)$, but we were not able to show this, other than the special case $x=n+1/2$, $n\in\Z$ in \eqref{bnd-ZxySpecial}. Such a general improvement would possibly require a finer Diophantine analysis on $x$ and $y$. See Figure \ref{fig:zxy} for a plot of $|S_1(x,y)|$, compared to this conjectural bound. This, along with the equality \eqref{bnd-ZxySpecial}, suggests that the generic bound given in \eqref{bnd-Zxy} is strong.
\begin{figure}[ht]
\centering
	\includegraphics[width=3in]{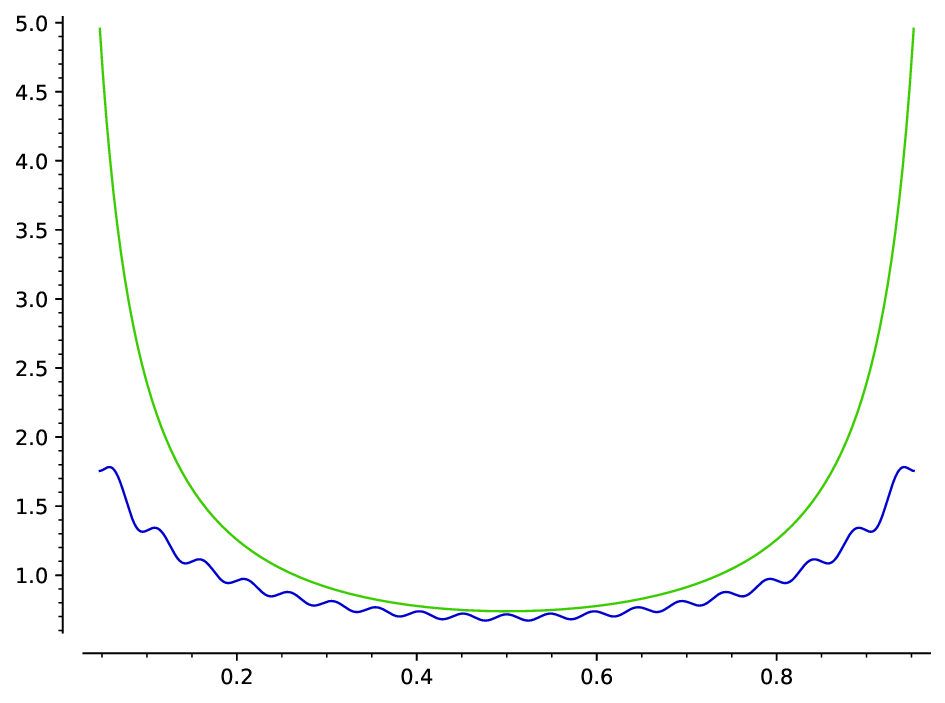}
	\caption{\label{fig:zxy}}Plot of $|S_1(x,y)|$ as a function of $x$, in the case $\lfloor y \rfloor = 21,$ in blue. In green, the conjectural bound $(1/y + \log 2)/|\sin(\pi x)|$, which we proved in the special case $x=n+1/2$, $n\in\Z.$
\end{figure}
\end{remark}
\subsection{Estimates of tails of exponential sums}
Let $N$ be a non-negative integer, and $x,y\in\R$ satisfying $x$ is not an integer and $y< N+1$.
\begin{lemma} \label{lemma_bnd_S_plus_minus}
Let $x,y\in\R$ with $1\leq y< N+1$, $x\notin\Z$,  and let $N$ be a non-negative integer. Then
    \begin{align}
    & \label{bnd-Sminus}
    \bigg|  \sum_{ \nu > N} \frac{e^{-2\pi i \nu x}}{\nu(\nu-y) }\bigg|  \leq Z_0(N,x,y) =  
   \frac{1}{y|\sin(\pi x)|} \left( \frac{1}{N+1-y} + \frac{1}{N+1} \right)
        ;\\
     & \label{bnd-Splus}
     \bigg| \sum_{ \nu=1}^{\infty} \frac{e^{-2\pi i \nu x}}{\nu(\nu+y) }\bigg|  \leq Z_1(x,y) =     
     \frac{1}{y|\sin(\pi x)|} \left( 1+\frac{1}{1+y} \right).
    \end{align}
    In the special case where $x\in\Z+\frac{1}{2}$ and $N+1-y \geq \frac{1}{2}$, the definitions can be replaced by the stronger estimates
    \begin{equation}
        \begin{split}
            &Z_0(N,x,y) =  \frac{1}{y}  \bigg| \frac{\Gamma'}{\Gamma}(N+1-y) - \frac{\Gamma'}{\Gamma}\left(\frac{N+2-y}{2}\right) - \log 2  \bigg|   +\frac{1}{y(N+1)} \leq \frac{1}{y}\left(\frac{\pi}{2} + \frac{1}{N+1} \right); 
\\
        &Z_1(x,y) =     
        \frac{\log 2}{y} + \frac{3}{2y(y+1)};\label{def:z01-halfInteger}
    \\
        \end{split}
    \end{equation}
\end{lemma}
\begin{proof}
We compare our sums with a Riemann--Stieltjes integral involving $S_0$ and
$S_1$, integrate by parts, and use the bounds of Lemma~\ref{lem:geometric}. Details can be found in Appendix \ref{secn:proofs-lemmas} of the ArXiv version of this paper. 
\end{proof}
\subsection{A first stationary phase estimate}
We first recall a bound on oscillatory integrals. Arias de Reyna gives an explicit version of \cite[Lemma 4.3]{Titchmarsh} with an optimal constant of $2$ (improving upon the constant $4$ in the original formulation):
\begin{lemma} \cite[Lemma 2]{dereyna2024approximationzetafunctiondirichlet}\label{Lemma 2 arias}
    Let $f(x)$ and $g(x)$ be continuous real functions defined in a closed interval $[a,b]$. Assume that $f$ has a non-null continuous derivative $f'(x)$ and that $\Big| \frac{g(x)}{f'(x)}\Big| $ is non-zero and monotonic. Then 
    \begin{equation}
        \bigg|\int_a^bg(x)e^{if(x)}dx\bigg|\le2\text{ max}_{a\leq x\leq b} \bigg| \frac{g(x)}{f'(x)}\bigg| 
    \end{equation}
\end{lemma}
For $m$ a positive integer, $s=\sigma + i t $ with $\sigma>0$, and $0\le a<b$ (with $b$ symbolically possibly being $\infty$), we introduce the notation 
\begin{equation}\label{def-intJabm}
J(a,b,m) = \int_a^b u^{-s} e^{2\pi i m u} \,du.    
\end{equation}
We apply the above lemma, and integration by parts, to 
estimate the following oscillatory sums and integrals. 
\begin{lemma}\label{lem:u-s}
    Let $x\ge 1$ with $x\in\Z+\frac{1}{2}$,  $s=\sigma+it$ with $0<\sigma<1$, and define $y$ by the equation $2\pi x y = |t|$. Assume additionally that $y\ge  1$ and $y\in \Z+\frac{1}{2}$. Let  $m\ge 1$ be an integer. 
    Then, we have the following estimates:
\begin{align}  
& \bigg|\sum_{1 \leq m < y }J(0,x,m)\bigg| \le x^{-\sigma} \Bigg( 
\frac{\log y }{ \pi } + \frac{1}{\pi}\left(\gamma+2\log 2-\frac{3}{2}\right)  +\frac{3}{4\pi y} + \frac{3}{8\pi y^2}
\Bigg),\label{est I2}
\\ &
\bigg|  \sum_{1 \leq m \leq y } J(N,\infty,m) \bigg|  
\leq 2 \frac{N^{-\sigma}}{\pi} \big(\log y + 1  \big)\ \text{ for all }\ N>\frac{t}{\pi m}.
 \label{est I3}
\end{align}
\end{lemma}
\begin{remark}
Note that the $(\log y)$-term arising from the $\sum_{1\le m<y} J(0,x,m)$ is the one appearing in the error term $\mathcal{E}$ in our main \thmref{thm-AFE2}, and \corref{cor-AFE2}.
\end{remark}
\begin{proof}
Integrating by parts twice and using that $x \in \Z+\frac{1}{2}$, so that $e^{2\pi i mx}=(-1)^m$,  we find that
\begin{equation}\label{ineq-J0xm}
        J(0,x,m) 
        = \frac{x^{1-s} (-1)^m}{1-s} - \frac{2\pi i m \, x^{2-s} (-1)^m}{(1-s)(2-s)} - \frac{4\pi^2 m^2}{(1-s)(2-s)} \int_0^x u^{2-s} e^{2\pi i m u} \,\mathrm{d}u.
\end{equation}
To bound the resulting integral, set $F(u)= 2\pi m u - t\log u$ and $G(u)=u^{2-\sigma}$. The ratio
\[
\left|\frac{G(u)}{F'(u)}\right| = \frac{u^{3-\sigma}}{|t|-2\pi m u}
\]
is non-negative and strictly increasing on $(0, x]$, so applying Lemma \ref{Lemma 2 arias} leads to:
\[
\left| \int_0^x u^{2-s} e^{2\pi i m u} \,\mathrm{d}u \right| 
\le 2 \frac{x^{3-\sigma}}{|t|-2\pi m x} 
= \frac{x^{2-\sigma}}{\pi(y - m)}.
\]
Using $|1-s||2-s| \ge |t|^2 = (2\pi x y)^2$, the integral expression is simplified as follows:
\[
\frac{4\pi^2 m^2}{|1-s||2-s|} \cdot \frac{x^{2-\sigma}}{\pi(y-m)} \le \frac{4\pi^2 m^2 x^{2-\sigma}}{(2\pi x y)^2 \pi (y-m)} 
= \frac{ m^2 x^{-\sigma} }{ y^2 \pi (y-m)} .
\]
Summing over $1 \le m < y$, and using
$\left|\sum_{m < y} (-1)^m\right| \le 1$, and $\left|\sum_{m < y} (-1)^m m\right| \le \frac{y+1}{2}$, and
\begin{equation}\label{bnd:m2y2y-m}
    \sum_{1 \leq m < y } \frac{m^2}{y^2(y- m) } \le  \log y 
+\gamma +2\log 2
- \frac{3}{2}
+ \frac{1}{8y^2},
\end{equation}
we obtain the bound for \eqref{ineq-J0xm}:
\[    \bigg|\sum_{1 \leq m < y }J(0,x,m)\bigg| 
    \le \frac{x^{-\sigma}}{2\pi  y} 
+ \frac{ x^{-\sigma}(y+1)}{4\pi y^2} 
+ \frac{ x^{-\sigma} }{ \pi } \left(  \log y  +\gamma +2\log 2
- \frac{3}{2}
+ \frac{1}{8y^2} \right) 
\]
which gives the announced bound. It remains to show the sum bound \eqref{bnd:m2y2y-m}. 
We can rewrite the summand 
\[ \frac{m^2}{y^2(y- m) }=\frac{1}{y-m}-\frac{1}{y}-\frac{m}{y^2} \] 
Therefore, the sum
\[\sum_{1 \leq m < y } \frac{m^2}{y^2(y- m) }= \sum_{1 \leq m < y }\frac{1}{y-m}-\frac{\floor{y}}{y}-\frac{\floor{y}(\floor{y}+1)}{2y^2}.
\]
Using properties of the digamma function (see Appendix \ref{sec:digamma} in the ArXiv version of this paper), we find 
\[
\sum_{1 \leq m < y }\frac{1}{y-m} 
= \frac{\Gamma'}{\Gamma}(y)-\frac{\Gamma'}{\Gamma}(y-\floor{y})
<
\log y -\frac1{2y} -\frac{\Gamma'}{\Gamma}(1/2).
\]
This implies that 
\begin{equation*}
    \sum_{1 \leq m < y } \frac{m^2}{y^2(y- m) }
\le
\log y -\frac1{2y} -\frac{\Gamma'}{\Gamma}(1/2)
- \frac{\floor{y}}{y}
- \frac{\floor{y}(\floor{y}+1)}{2y^2}.
\end{equation*}
Using that $\floor{y}=y-\frac{1}{2}$ and $-\frac{\Gamma'}{\Gamma}(1/2) = \gamma+2\log 2$, we obtain the announced bound \eqref{bnd:m2y2y-m} upon simplification.
\smallskip 

We now apply Lemma \ref{Lemma 2 arias} to bound $J(N,\infty,m)$\footnote{Strictly speaking, we apply Lemma \ref{Lemma 2 arias} to the closed interval $[N,N_1]$ and then let $N_1\to\infty$.}. For $F(u)= 2\pi mu-t\log u$ and $G(u)=u^{-s}$, 
we observe that $\big| \frac{G(u)}{F'(u)}\big|
=\frac{|u^{-s}|}{|2\pi m-\frac{t}{u}|}$ 
is positive and decreasing for $0<\sigma<1$ and $u \ge N > \frac{t}{\pi m}$. Thus
\begin{equation}  
    |J(N,\infty,m)|
    \le 2 \max_{u \geq N} \left(\frac{u^{-\sigma}}{2\pi m - \frac{t}{u}}\right) 
    = \frac{2 N^{-\sigma}}{2\pi m - \frac{t}{N}}
     \le2 \frac{N^{-\sigma}}{\pi m}.
\end{equation}
We conclude to \eqref{est I3} by using $ \sum_{1 \leq m \leq  y } \frac{1}{m} \leq \log y + 1$. 
\end{proof}
Finally, the following results will be useful for bounding components appearing in the functional equation of the Riemann zeta function.
\subsection{Approximations for $\chi$ and $\Gamma$}\label{sec:chi}
The $\chi$ function provides a simple way of expressing the functional equation of the Riemann zeta function. 
For all complex value $s$,
\[\zeta(1-s) = \chi(s) \zeta(s),\]
where $\chi$ is defined as 
\begin{equation}
    \chi(s) = 2^s \pi^{s-1} \Gamma(1-s) \sin\left(\frac{\pi s}{2}\right), \label{def:chi}
\end{equation}
and satisfies (see \cite[Eq. 4.12.3]{Titchmarsh}), as $t\to\infty$: 
\begin{equation} 
   \chi(s)\sim \tilde{\chi}(s),\ \ \text{   where} \ \ \ \tilde{\chi}(\sigma+it) =  \left(
    \frac{2\pi}{|t|}
    \right)^{\sigma-\frac{1}{2}} 
    \left(
    \frac{|t|}{2\pi e}
    \right)^{-it}
    e^{\text{sgn}(t)\frac{\pi}{4}i }. \label{eq:chi_asymptotic}
\end{equation}
We recall here the explicit version from \cite[Proposition 1] {Simonic2020}. 
\begin{lemma}\label{lem:bndchi-Simonic}
Let $s=\sigma+it$ with $1/2 \leq  \sigma \leq 1$ and $|t| \geq t_0 \geq 1/\pi.$ Then
\begin{equation}
|\chi(s)| \leq 
 C_0(\sigma,t_0) \left( \frac{2\pi}{|t|} \right)^{\sigma - 1/2} ,
\label{eq:chiBound}
\end{equation}
with 
\begin{align}
    \label{def-C0}
    & C_0(\sigma,t_0) = 
    1 + \frac{1}{t_0} \left(
    C_1(\sigma, t_0) \big(1 + e^{-\pi t_0}\big) C_2(t_0) + C_3(t_0)
    \right), 
\\
& C_1(\sigma,t_0) = (1-\sigma)^2 \left(\frac{1}{2}+\frac{2}{\pi}\right) + (1-\sigma)(\sigma-\frac{1}{2})\left(\left(\frac{\pi}{2}\right)^2+\frac{1-\sigma}{2t_0}\right)   ,\\
& C_2(t_0) = \exp\left(\frac{1}{12t_0}+\frac{1}{90t_0^3}\right),  \ 
C_3(t_0) = \frac{C_2(t_0)-1}{\log C_2(t_0)}\left(\frac{1}{12}+\frac{1}{90t_0^2}\right)+t_0e^{-\pi t_0}C_2(t_0) .\label{eq:Ci} 
\end{align}
\end{lemma}
Rearranging definition \eqref{def:chi} yields the following expression for $\Gamma(1-s)$:
\begin{lemma}\label{lem:gamma-chi2}
Let $s=\sigma+it$ with $0\leq \sigma \leq 1$ and $|t|\geq t_0\geq \frac{1}{\pi}$. We have \begin{equation*}
    \Gamma(1-s) \left(\dfrac{2\pi}{i}\right)^{s-1} = \chi(s) \left(1 + \mathcal{O}^*  \bigg(\dfrac{e^{-\pi t}}{1 - e^{-\pi t_0}}\right)\bigg) . 
\end{equation*}
\end{lemma}
\section{An explicit van der Corput approximate Poisson summation formula}
\label{Section-Theorem-Bestimate}
Our first main theorem provides  estimates that are essential for evaluating the sums found in the explicit Approximate Functional Equations (\thmref{thm-AFE1} and \thmref{thm-AFE2}). 
This explicit result can also be of independent interest and applied in other contexts.
\subsection{Statement of results}
\begin{theorem} \label{thm-VDC}\ 
    \begin{enumerate}
    \item[{\bf Part I.}] Let $f(x)$ be a real function with continuous, positive, and strictly decreasing derivative  $f'(x)$ in $[a,b]$. Let $g(x)$ be a real positive decreasing function with a continuous derivative $g'(x)$. Let $N$ be a non-negative integer with $N< f'(b)$, and define $\delta = 1 - (f'(a)-\lfloor f'(a) \rfloor)$, we have:
    \begin{equation}
        \label{vanderCorputB}
        \sum_{a<n \leq b}g(x) e^{2\pi i f(n)} = 
    R_N(a,b) + \mathcal{O}^*\left(T_N(a,b)\right),
    \end{equation}
    where
\begin{align} 
\label{def-RNab}
R_N(a,b) =& \sum_{\nu =N}^{\lfloor f'(a) \rfloor} \int_{a}^{b} g(x)e^{2\pi i (f(x)-\nu x)} dx , 
\\
T_N(a,b) = & \frac{|g'(a)|+2\pi g(a) \left(f'(a)-N\right)}{2\pi^2 \left(f'(a)-N\right)} \bigg(\log(1+f'(a)-N)+\gamma +\log(1+\lfloor f'(a) \rfloor) - \frac{\Gamma'}{\Gamma}(\delta) 
\bigg. \nonumber
\\ \bigg.& -\frac1{2(1+\lfloor f'(a) \rfloor )}  -\frac{1}{2(1+f'(a)-N)}\bigg)  + \frac{g(b) |\tilde{S_1}(b,f'(a)-N)| + g(a) |\tilde{S_1}(a,f'(a)-N)|}{2\pi} \bigg. \nonumber
\\ \bigg.  &+ G(a,b) ,\label{def-TNab}
\end{align}
and $G$ and $\tilde{S_1}$ are respectively defined in \eqref{bigO1} and Lemma \ref{lem:geometric}.

In particular, if $a,b\in\mathbb{Z}+\frac12$, the definition of $T_N(a,b)$ simplifies to
\begin{equation}\label{def-TNab-integer+1/2}
    \begin{aligned}
       T_N(a,b) =& \frac{|g'(a)|+2\pi g(a) \left(f'(a)-N\right)}{2\pi^2 \left(f'(a)-N\right)} \bigg(\log(1+f'(a)-N)+\gamma +\log(1+\lfloor f'(a) \rfloor) - \frac{\Gamma'}{\Gamma}(\delta) 
        \bigg. \\ \bigg.& -\frac1{2(1+\lfloor f'(a) \rfloor)}  -\frac{1}{2(1+f'(a)-N)}\bigg)  + \frac{ g(a)+g(b)  }{2\pi}  \left(\log 2 + \frac{1}{f'(a)-N} \right).    \end{aligned}
\end{equation}
\item[{\bf Part II.}] If, alongside the assumptions of Part I, we assume the functions $g$ and $f$ are in $C^2(a,b)$ and $|g'|$, $|f''|$, $g''$, and $|g'f'+gf''|$ are all positive and decreasing on $[a,b]$, then the definition of $T_N(a,b)$ simplifies to:
\begin{equation}\label{def-TNab-partII}
\begin{aligned}
  T_N(a, b) = &\frac{g(b) |\tilde{S}_1(b, f'(a)-N)| + g(a) |\tilde{S}_1(a, f'(a) -N)|}{2\pi}+G(a,b)  \\
    &+ \frac{H(b)\mathcal{B}_{f'}(b) + H(a)\mathcal{B}_{f'}(a)}{4\pi^2} 
    + \frac{H_1(a) \mathcal{E}_1 (\delta,f'(a))+ H(a)|f''(a)| \mathcal{E}_2(\delta,f'(a))}{4\pi^3f'(a)}, 
    \end{aligned}
\end{equation}
where $\gamma$ is Euler's constant, and the bounds $G$ and $\tilde{S_1}$ are respectively defined in \eqref{bigO1} and Lemma \ref{lem:geometric}. Here, $H, H_1, \mathcal{B}_{f'}, \mathcal{E}_1$, and $\mathcal{E}_2$ are respectively defined in \eqref{eq:defH(x)}, \eqref{eq:defH_1(x)}, \eqref{def-B},\eqref{auxilliary_error1}, and \eqref{auxillary_error2}.

In particular, if $a,b\in\mathbb{Z}+\frac12$
\eqref{def-TNab-partII} simplifies to
\begin{equation}\label{def-TNab-integer+1/2-partII}
    \begin{split}
       T_N(a,b) =& \Big(\frac{\log 2}{2\pi} + \frac{1}{2\pi \left(f'(a)-N\right)}\Big) \left( g(b) + g(a) \right)  
        + \frac{H(b)}{4\pi^2} \mathcal{B}_{f'}(b) + \frac{H(a)}{4\pi^2} \mathcal{B}_{f'}(a) 
        \\& + \frac{H_1(a)}{4\pi^3f'(a)} \mathcal{E}_1(\delta,f'(a))+ \frac{H(a)|f''(a)|}{4\pi^3f'(a)} \mathcal{E}_2(\delta,f'(a)). 
\end{split} 
\end{equation}
\end{enumerate}
\end{theorem}
\begin{remark}
For simplicity, we stated the case where $N=0$ in Corollary \ref{cor-VDC} in the Introduction. The general case for $N< f'(b)$ is easily recovered without loss of generality by replacing $f(x)$ with $f(x)-Nx$. The condition $N< f'(b)$ is required so that $f(x)-Nx$ has a positive derivative. 
\end{remark}
\begin{corollary}\label{cor-thmVDC-partII}
Under the same assumptions as in \thmref{thm-VDC} Part-II, and for $\delta=\frac{1}{2}$, we have:
\begin{align} 
\label{def:B_f(Z+1/2)}
\mathcal{B}_{f'}(x) =&  \frac{1}{f'(x)}\Big( \frac{\pi}{2}  + \frac{1}{(\lfloor f'(a) \rfloor-N+1)} + \log 2+\frac{3}{2(f'(x)+1)}\Big),
\\
\label{def:E1(Z+1/2)}
\mathcal{E}_1 \left(1/2, f'(a)\right)=&
\frac{46}{9} - \frac{1}{f'(a)-N}\Big(\log(\lfloor f'(a) \rfloor -N+ 1) - \frac{1}{\lfloor f'(a) \rfloor-N + 1} - \frac{\Gamma'}{\Gamma} (1/2) \Big.
\\& + \Big.\log(f'(a) + 1) + \gamma - \frac{1+2(f'(a)-N)}{2(1+f'(a)-N)}\Big),
\\
\label{def:E2(Z+1/2)}
\mathcal{E}_2\left(1/2, f'(a)\right)=&
\frac{230}{27} - \frac{14}{3(f'(a)-N)} +\frac{1}{(f'(a)-N)^2}\Big( \log(\lfloor f'(a) \rfloor-N + 1) - \frac{1}{2(\lfloor f'(a) \rfloor-N + 1)} 
\Big.\\ &\Big.
- \frac{\Gamma'}{\Gamma}(1/2) +\log(f'(a)-N + 1) + \gamma - \frac{3(f'(a)-N) + 3(f'(a)-N)^2 + 1}{2(1+(f'(a)-N))^2}\Big).
\end{align}  
\end{corollary}
\begin{remark}[Comparison with similar Theorems]
\thmref{thm-VDC} and \corref{cor-thmVDC-partII} improve Arias de Reyna'\cite[Lemmas 4 and 5]{dereyna2024approximationzetafunctiondirichlet}.
Lemma 4 establishes the estimate \eqref{vanderCorputB} in the case $g(n)=1$, with 
\[
T_0(a,b) =  \frac{ G}{\pi}\Big(3\log(1+f'(a))+\pi+3\gamma+\frac{1}{\delta}\Big) \ \text{and}\ G= 1.
\]
Lemma 5 generalizes this via partial summation, with
\[
G= |g(b)|+\int_a^b |g'(x)| dx=g(a)
\]
when $g$ is positive and decreasing. For instance, for $\s\geq 0$, $G=a^{-\sigma}$ when $g(u) = u^{-\s}$. 
Thus, for $f'(a)$ large enough, 
\[ \lim_{b\to\infty} T_0(a,b) \approx  C \log(1+f'(a) ) a^{-\sigma}
\] 
with \[
C=\frac3{\pi}, \frac2{\pi}, \text{ and }0
\] respectively in \cite[Lemma 5]{dereyna2024approximationzetafunctiondirichlet}, Part I and Part II of  \thmref{thm-VDC} (choosing $\delta $ away from $0$). 
\end{remark}
\subsection{Proof of \thmref{thm-VDC}}\label{Section-Proof-Theorem-Bestimate}
\begin{proof} 
The primary objective of this proof is to approximate the sum by an integral and rigorously bound the resulting error terms. 
We may assume without loss of generality that the non-negative integer $N$ is equal to zero. This is a valid assumption because the original inequality remains unchanged if we replace $f(x)$ with $f(x)-Nx$ and $N$ with $0$. Note that this substitution does not affect the derivative. This implies that the floor function of the absolute value of the new derivative, $\lfloor |f'(x)-N| \rfloor$, simplifies to $\lfloor |f'(x)| \rfloor-N$, which is non-negative.

\smallskip
The Euler-Maclaurin formula relates a sum to an integral:
\begin{multline}
\label{Stieltjes}
    \sum_{a<n \leq b}g(x) e^{2\pi i f(n)} - \int_{a}^{b} g(x)e^{2\pi i f(x)} dx  =\bigg(\lfloor b \rfloor - b + \frac{1}{2}\bigg)g(b)e^{2\pi i f(b)} - \bigg(\lfloor a \rfloor - a + \frac{1}{2}\bigg)g(a)e^{2\pi i f(a)} \\
    -\int_{a}^{b} \bigg(\lfloor x \rfloor - x + \frac{1}{2}\bigg)(g'(x)+2\pi i g(x)f'(x))e^{2\pi i f(x)} dx.
\end{multline}
We apply a trivial bound to the first two boundary terms via the triangle inequality, defining $G(a,b)$ as:
\begin{equation}\label{bigO1}
\begin{split}
G(a,b) &= \bigg| \left(\lfloor b \rfloor - b + \frac{1}{2}\right)g(b)e^{2\pi i f(b)} - \left(\lfloor a \rfloor - a + \frac{1}{2}\right)g(a)e^{2\pi i f(a)}\bigg|  \\
&= \begin{cases}
\mathcal{O}^*\left( \frac{g(b)+ g(a)}{2} \right), & \\ 
0 & \text{if } a,b \in \mathbb{Z}+\frac{1}{2}.
\end{cases}
\end{split}
\end{equation}
Next, we substitute the Fourier series representation of $( \lfloor x \rfloor - x + \frac{1}{2}  )$ into the left integral of \eqref{Stieltjes}. 
\begin{equation}
- \Big( \lfloor x \rfloor - x + \frac{1}{2} \Big)=  \frac{1}{\pi} \sum_{\nu \ge1} \frac{\sin(2\pi \nu x)}{\nu } = \frac1{2\pi i}\sum _{\nu =1}^{\infty}\big(e^{-2\pi i\nu x}
    - e^{2\pi i\nu x}\big) .\label{fourier}
 \end{equation}
 Applying the Dominated Convergence Theorem justifies interchanging the order of summation and integration. By combining the boundary term estimate from \eqref{bigO1} with the Fourier expansion \eqref{fourier}, we rewrite \eqref{Stieltjes} as:
    \begin{equation}
    \label{eqn-exp}
        \sum_{a<n \leq b}g(x) e^{2\pi i f(n)} - \int_{a}^{b} g(x)e^{2\pi i f(x)} dx 
        = S_1-S_2 +G(a,b), 
    \end{equation}
    where $G(
    a,b)$ is defined in \eqref{bigO1}, and 
\begin{align}
    S_1 & =\sum_{\nu=1}^{\infty}\frac{1}{2\pi i \nu}\int_a^b\left(g
    '(x)+2\pi ig(x)f'(x)\right)e^{2\pi i\left(f(x)-\nu x\right)}dx,\label{def-S1}\\
    S_2 & = \sum_{\nu=1}^{\infty}\frac{1}{2\pi i \nu}\int_a^b\left(g
    '(x)+2\pi ig(x)f'(x)\right)e^{2\pi i\left(f(x)+\nu x\right)}dx. \label{def-S2}
\end{align}
The behavior of each integral in the sum $S_1$ depends critically on the sign of the derivative of the exponent, given by $f'(x)-v$. Since $f'(x)$ is a strictly decreasing function, we must partition the sum into two distinct parts at the point where this term might change its sign. Therefore,
we truncate the sum $S_1$ at $\lfloor f'(a)\rfloor$:
\begin{equation}
    \label{S1-S11-S12}
    S_1=S_{11}+S_{12},
\end{equation}
where 
\begin{equation}
    S_{11}  = \sum_{1 \leq \nu  \le\lfloor f'(a) \rfloor}\frac{1}{2\pi i \nu}\int_a^b\left(g
    '(x)+2\pi ig(x)f'(x)\right)e^{2\pi i\left(f(x)-\nu x\right)}dx,\label{def-S11}
    \end{equation}
    and
    \begin{equation}
    S_{12}  = \sum_{ \nu  > \lfloor f'(a) \rfloor} \frac{1}{2\pi i \nu}\int_a^b\left(g
    '(x)+2\pi ig(x)f'(x)\right)e^{2\pi i\left(f(x)-\nu x\right)}dx. \label{def-S12}
\end{equation}
Rewriting and using integration by parts, we get:
 \begin{equation}
S_{11} =
 \sum_{1 \leq \nu  \le\lfloor f'(a) \rfloor} \frac1{\nu}  \int_a^b e^{-2\pi i \nu x } d\left( \frac{g(x) e^{2\pi i f(x)}}{2\pi i} \right)  =   \Sigma_{11}(b)-\Sigma_{11}(a)   +\Sigma_{12} , \label{S_11}
  \end{equation}
        where
\begin{equation}
    \Sigma_{11}(x)
    =  \sum_{1\leq \nu \leq \floor{f'(a)}} \frac{g(x) e^{2\pi i (f(x)-\nu x)}}{2\pi i \nu } 
    = \frac{g(x)e^{2\pi i f(x)}}{2\pi i} {S_1}(x,\floor{f'(a)})\label{def-Sigma11}
\end{equation}
with $S_1(x,\floor f'(a))$ defined in  \eqref{def-Sxy}, and
\begin{equation}
    \Sigma_{12} =   \sum_{1 \leq \nu \le\lfloor f'(a) \rfloor} \int_a^b g(x)e^{2\pi i(f(x)-\nu x)} dx .\label{def-Sigma12}
\end{equation}
Together with \eqref{eqn-exp}, \eqref{S1-S11-S12}, and \eqref{S_11}, we have:
\begin{multline}
       \sum_{a<n \leq b}g(x) e^{2\pi i f(n)} - \int_{a}^{b} g(x)e^{2\pi i f(x)} dx = \left( \Sigma_{11}(b)-\Sigma_{11}(a)\right)+\Sigma_{12}+S_{12}-S_2 +G(a,b). \label{sumS1-S2+G} 
\end{multline}
From the definition \eqref{def-Sigma12} of $\Sigma_{12}$, we observe that 
\begin{equation}
    \label{MainSigma12}
    \int_{a}^{b} g(x)e^{2\pi i f(x)} dx  + \Sigma_{12} 
    = \sum_{\nu =0}^{\lfloor f'(a) \rfloor} \int_a^b g
    (x)e^{2\pi i(f(x)-\nu x)} dx ,
\end{equation}
so that with \eqref{sumS1-S2+G},
\begin{multline}
    \sum_{a<n \leq b}g(x) e^{2\pi i f(n)} 
    = \sum_{\nu =0}^{\lfloor f'(a) \rfloor} \int_a^b g
    (x)e^{2\pi i(f(x)-\nu x)} dx  + \left( \Sigma_{11}(b)-\Sigma_{11}(a)\right)+S_{12}-S_2 +G(a,b). \label{SumS1-S2+G_nu=0}
\end{multline}
Applying \eqref{bnd-Zxy} from Lemma \ref{lem:geometric} to evaluate $\Sigma_{11}$:
\begin{equation}
    |\Sigma_{11}(b)-\Sigma_{11}(a)|
    \leq \frac{1}{2\pi } \left( g(b) |\tilde{S_1}(b,\lfloor f'(a) \rfloor)| + g(a)|\tilde{S_1}(a,\lfloor f'(a) \rfloor)| \right).
    \label{bound-sigma11ab}
\end{equation}
Denoting 
\begin{equation}\label{def-R0ab}
    R_0(a,b) = \sum_{\nu =0}^{\lfloor f'(a) \rfloor} \int_{a}^{b} g(x)e^{2\pi i (f(x)-\nu x)} dx ,
\end{equation}
we conclude with \eqref{SumS1-S2+G_nu=0} that
 \begin{equation}
    \label{eq:defsum(atob)}
        \bigg| \sum_{a<n \leq b}g(x) e^{2\pi i f(n)}\!
       - R_0(a,b)\bigg| 
        \!\leq  \frac{1}{2\pi } \left( g(b) |\tilde{S_1}(b,\lfloor f'(a) \rfloor)| + g(a)|\tilde{S_1}(a,\lfloor f'(a) \rfloor)| \right)   +|S_{12}|+|S_2|
        +G(a,b) .
        \end{equation}

\smallskip
We are separating the cases for study of $S_{12}$ and $S_2$ for part I. and part II. to give better approximations respectively when $f'(a)$ is small (which is the case in AFE1) and large (in the case of AFE2).

\subsubsection{Proof of Part I. of \thmref{thm-VDC}}

\begin{itemize}
\item   {\it Study of $S_{12}$:}\\
From the definition \eqref{def-S12}:
\begin{equation} \label{eq:defS_12}
\begin{split}
    S_{12}  &=\sum_{ \nu  > \lfloor f'(a) \rfloor}\frac{1}{2\pi i \nu}\left(I_{g'}(-\nu)+2\pi i I_{gf'}(-\nu)\right) .
    \end{split}
\end{equation}
where,
\begin{equation}
    I_h(-\nu)=\int_a^bh(x)e^{2\pi i \left(f(x)-\nu x\right)}dx \label{Ih}
\end{equation}

For $\nu  > \lfloor f'(a) \rfloor$, the function $\frac{h(x)}{\pi(\nu -f'(x))}$ is decreasing and positive.
Applying Lemma \ref{Lemma 2 arias} on integral $I_h(-\nu)$ defined in \eqref{Ih} gives:
\begin{equation} \label{bndI_h(-nu)}|I_h(-\nu)| \leq \max_{a\leq x\leq b} \bigg| \frac{h(x)}{\pi\left(f'(x)-\nu\right)}\bigg|  = \frac{|h(a)|}{\pi(\nu -f'(a))} .
\end{equation}
By substituting $h(x)$ with $g'(x)$ and $g(x)f'(x)$ in \eqref{bndI_h(-nu)}, we obtain a bound for the integral component $I_{g'}$ and $I_{gf'}$ in \eqref{eq:defS_12}.
This allows us to bound the second error sum:
\begin{equation}
    \label{bnd1-S12}
    \begin{split}
    |S_{12}|&\leq \sum_{\nu >\lfloor f'(a) \rfloor} \frac{1}{2\pi \nu } \Big( \big| ( I_{g'}(\nu) \big| +2\pi   \big| I_{gf'}(\nu)\big| \Big) 
    \\&\leq  \frac{ |g'(a)|+2\pi g(a)f'(a)}{2\pi^2}  \sum_{\nu >\lfloor f'(a) \rfloor} \frac{1}{ \nu (\nu -f'(a))} 
\end{split}
\end{equation} 
The sum $\sum_{\nu >\lfloor f'(a) \rfloor} \frac{1}{ \nu (\nu -f'(a))}$ can be bounded using \eqref{bnd-sum-nu(nu-y)blower-bupper} from Lemma \ref{lem:harmonic-lead}
\begin{equation}
\label{rewriteS12}
    \sum_{\nu>\lfloor f'(a) \rfloor}\frac{1}{\nu(\nu-f'(a))} \leq \frac1{f'(a)} \left( \log(\lfloor f'(a) \rfloor+1)-\frac1{2(\lfloor f'(a) \rfloor+1)} 
- \frac{\Gamma'}{\Gamma}(\delta)\right)  ,
\end{equation}
where $\delta = 1-(f'(a)-\lfloor f'(a) \rfloor)$. 
We conclude
\begin{equation}
    \label{bnd-S_12}
    |S_{12}|
    \leq \frac{ |g'(a)|+2\pi g(a)f'(a)}{2\pi^2 f'(a)}  \left( \log(\lfloor f'(a) \rfloor+1)-\frac1{2(\lfloor f'(a) \rfloor+1)} 
- \frac{\Gamma'}{\Gamma}(\delta)\right) . 
\end{equation} 

\item {\it Study of $S_2$:}  

Consider the sum from \eqref{def-S11}\begin{equation} \label{sum_S_2_Ig'_I_(gf')}
    S_2=\sum_{\nu=1}^{\infty}\frac{1}{2\pi i \nu}\left(I_{g'}(+\nu)+2\pi i I_{gf'}(+\nu)\right) ,
    \end{equation}
where
\begin{equation}
    I_h(+\nu)=\int_a^bh(x)e^{2\pi i (f(x)+\nu x)}dx,
\end{equation}
    Here, the function $f'(x)+\nu $, is non-null and continuous and $\bigg| \frac{h(x)}{\pi (f'(x)+\nu )}\bigg| $ is also non-zero and decreasing for $h= g' \text{ and } h =gf'$. So, by Lemma \ref{Lemma 2 arias}:
    \begin{equation} \label{I_h}
    |I_h(+\nu)|\le\max_{a\leq x \leq b}\frac{h(x)}{\pi(f'(x)+\nu )}=\frac{|h(a)|}{\pi(f'(a)+\nu )}.
    \end{equation}
    Thus, $S_2$ as defined in \eqref{sum_S_2_Ig'_I_(gf')} can be bounded by using \eqref{I_h} and \eqref{eq:sum_infty_power_k}.
 \begin{equation}
     |S_2|  \le\frac{1}{2\pi^2}\sum_{\nu \ge1}\frac{|g'(a)|+2\pi g(a)f'(a)}{\nu (f'(a)+\nu)}\leq \frac{|g'(a)|+2\pi g(a)f'(a)}{2\pi^2 f'(a)}\left(\gamma+\log(1+f'(a))-\frac{1}{2(f'(a)+1)}\right).
    \label{bnd-S_2}
\end{equation}
\end{itemize}
Finally, we combine the bounds for $S_{12}$, $S_2$ from \eqref{bnd-S_12} and \eqref{bnd-S_2} and substitute them in \eqref{eq:defsum(atob)}:
\begin{equation}\label{final-R0T0part1}
\sum_{a<n \leq b} g(x)e^{2\pi i f(n)} =
      R_0(a,b) + \mathcal{O}^*\left(T_0(a,b)\right),
    \end{equation}
    where $R_0$ is defined in \eqref{def-R0ab} and 
\begin{multline}\label{def-T0ab}
        T_0(a,b) 
  =  \frac{g(b) |\tilde{S_1}(b,f'(a))| + g(a) |\tilde{S_1}(a,f'(a))|}{2\pi}   + G(a,b)  
        \\  + \frac{|g'(a)|+2\pi g(a)f'(a)}{2\pi^2 f'(a)}\left(\gamma+ \log(1+f'(a))+\log(\lfloor f'(a) \rfloor+1)-\frac1{2(\lfloor f'(a) \rfloor+1)} -\frac{1}{2(f'(a)+1)}
- \frac{\Gamma'}{\Gamma}(\delta)\right)  .
\end{multline}

\subsubsection{ Proof of Part II. of \thmref{thm-VDC}}

Starting from \eqref{eq:defsum(atob)}, the proof differs from Part I in the estimate for $S_{12}$, and consequently $S_2$ and $T_0$. We rewrite the definition of $I_h(-\nu)$ from \eqref{Ih}:
\begin{equation}
     I_{h}(-\nu)
    =\int_a^b \frac{h(x)}{2\pi i(f'(x)-\nu)} \left(2\pi i(f'(x)-\nu) e^{2\pi i \left(f(x)-\nu x\right)}\right)dx.
\end{equation}
We integrate by parts and obtain:
\begin{multline}
    \bigg| I_h(-\nu) -\frac{1}{2\pi}   \bigg( \frac{h(b) e^{2 \pi i (f(b) - \nu b)}}{f'(b) - \nu} - \frac{h(a) e^{2 \pi i (f(a) - \nu a)}}{f'(a) - \nu}\bigg) \bigg| \leq \frac{1}{2\pi} \bigg|  \int_a^b \frac{h'(x)}{f'(x) - \nu} e^{2\pi i (f(x) - \nu x)} dx \bigg|  \\+ \frac{1}{2\pi} \bigg|  \int_a^b \frac{h(x)f''(x)}{(f'(x) - \nu)^2} e^{2\pi i (f(x) - \nu x)} dx \bigg| \label{bnd1-|I_h(-nu)|}.
\end{multline}
Since $|h'(x)|$ and $|h(x)f''(x)|$ are positive and decreasing for the choices $h=g'$ and $h=gf'$, we analyze the monotonicity of the integrands. Let $G(x)=\frac{h'(x)}{f'(x)-\nu}$ or $G(x)=\frac{h(x)f''(x)}{f'(x)-\nu}$, and $F'(x)=2\pi(f'(x)-\nu)$. The corresponding quotients 
\begin{equation*}
    \frac{|h'(x)|}{2\pi(\nu-f'(x))^2} \quad \text{and} \quad \frac{|h(x)f''(x)|}{2\pi(\nu-f'(x))^3}
\end{equation*}
are strictly monotonic, both achieving their maximum values at $x=a$.
Applying Lemma \ref{Lemma 2 arias}, we obtain the following bounds for the integrals:
\begin{equation}
\label{bnd-int-h1}
     \bigg|  \int_a^b \frac{h'(x)}{f'(x)-\nu} e^{2\pi i \left(f(x)-\nu x\right)}\,dx \bigg|  \leq 2 \bigg| \frac{G(a)}{F'(a)}\bigg|  = \frac{|h'(a)|}{\pi(\nu-f'(a))^2},
\end{equation}
and
\begin{equation}
\label{bnd-int-h2}
     \bigg|  \int_a^b \frac{h(x)f''(x)}{(f'(x)-\nu)^2} e^{2\pi i \left(f(x)-\nu x\right)}\,dx \bigg|  \leq  \frac{|h(a)f''(a)|}{\pi(\nu-f'(a))^3}.
\end{equation}
Substituting the bounds from \eqref{bnd-int-h1} and \eqref{bnd-int-h2} into \eqref{bnd1-|I_h(-nu)|}:
\begin{equation}
    \bigg| I_h(-\nu) -\frac{1}{2\pi} \bigg( \frac{h(b) e^{2 \pi i (f(b) - \nu b)}}{f'(b) - \nu} - \frac{h(a) e^{2 \pi i (f(a) - \nu a)}}{f'(a) - \nu}\bigg)\bigg|  \leq \frac{1}{2\pi^2} \frac{|h'(a)|}{(\nu-f'(a))^2}+ \frac{1}{2\pi^2} \frac{|h(a)f''(a)|}{(\nu-f'(a))^3} .\label{final-bnd-|I_h(-nu)|}
\end{equation}
To bound $S_{12}$ as given in \eqref{eq:defS_12} using \eqref{final-bnd-|I_h(-nu)|}, we sum over $\nu$:
\begin{equation}
\begin{aligned}
\bigg| \sum_{\nu > \lfloor f'(a) \rfloor} \frac{1}{2 \pi \nu} I_h(-\nu)\bigg|  
&\le\bigg|  \frac{h(b) }{4 \pi^2 \nu} \sum_{\nu > \lfloor f'(a) \rfloor}  \frac{e^{2 \pi i (f(b) - \nu b)}}{f'(b) - \nu} \bigg| +\bigg| \frac{h(a) }{4 \pi^2 \nu} \sum_{\nu > \lfloor f'(a) \rfloor}\frac{ e^{2 \pi i (f(a) - \nu a)}}{f'(a) - \nu} \bigg|  \\
& + \frac{|h'(a)|}{4 \pi^3} \sum_{\nu > \lfloor f'(a) \rfloor} \frac{1}{\nu (\nu - f'(a))^2}+ \frac{|h(a) f''(a)|}{4 \pi^3} \sum_{\nu > \lfloor f'(a) \rfloor} \frac{1}{\nu (\nu - f'(a))^3}.
\end{aligned} \label{eq:S12_bound_full}
\end{equation}
We use \eqref{bnd-Sminus} in Lemma \ref{lemma_bnd_S_plus_minus} to bound the following sums:
\begin{align}
\bigg|  \frac{h(b)}{4\pi^2} \sum_{\nu > \lfloor f'(a) \rfloor} \frac{e^{2\pi i (f(b)-\nu b)}}{\nu(f'(b)-\nu)} \bigg|  & \leq \frac{|h(b)|}{4\pi^2 } Z_0( \lfloor f'(a)\rfloor, b, f'(b)) , \label{bnd_sum_h(b)_f(b)} \\
\bigg|  \frac{h(a)}{4\pi^2} \sum_{\nu > \lfloor f'(a) \rfloor} \frac{e^{2\pi i (f(a)-\nu a)}}{\nu(f'(a)-\nu)} \bigg|  
&\leq \frac{|h(a)|}{4\pi^2} Z_0(\lfloor f'(a) \rfloor, a, f'(a)). \label{bnd_sum_h(a)_f(a)}
\end{align}
Substituting the estimates for sums \eqref{bnd-sum-nu(nu-y)^2bupper} and \eqref{bnd-sum-nu(nu-y)^3bupper} in Lemma \ref{lem:harmonic-lead} into the remaining terms of \eqref{eq:S12_bound_full}, we obtain the following bounds for the error sums over $\nu > \lfloor f'(a) \rfloor$
\begin{multline}
  \frac{|h'(a)|}{4\pi^3 }\sum_{ \nu  > \lfloor f'(a) \rfloor} \frac{1}{\nu(\nu-f'(a))^2} 
\le
  \frac{|h'(a)|}{4\pi^3 f'(a)}  \left[ \left(\frac1{\delta^2}+\frac1{(\delta+1)^2}+ \frac1{\delta+1}  \right)\right. \\
   \left. -\frac1{f'(a) }  \left( \log(\lfloor f'(a) \rfloor+1)-\frac1{\lfloor f'(a) \rfloor+1} - \frac{\Gamma'}{\Gamma}(\delta)\right)\right], \label{error_sum_(nu-f'(a))^2}
\end{multline}
and
\begin{multline}
  \frac{|h(a)f''(a)|}{4\pi^3 }\sum_{ \nu  > \lfloor f'(a) \rfloor} \frac{1}{\nu(\nu-f'(a))^3} 
\leq  \frac{|h(a)f''(a)|}{4\pi^3 f'(a)}
\left[ 
\frac1{\delta^3}+\frac1{(\delta+1)^3}+\frac1{2(\delta+1)^2} \right.\\
- \frac1{f'(a)}\left( \frac1{\delta^2} + \frac1{\delta+1} \right) 
\left. +\frac1{f'(a)^2}\left( \log(\lfloor f'(a) \rfloor+1)-\frac1{2(\lfloor f'(a) \rfloor+1)} - \frac{\Gamma'}{\Gamma}(\delta) \right)\right] . \label{error_sum_(nu-f'(a)^)3}
\end{multline}
Representing the components of $S_{12}$ from \eqref{eq:defS_12} in terms of the general integral $I_h(-\nu)$ for $h=g'$ and $h=gf'$, we obtain  from \eqref{bnd_sum_h(b)_f(b)}, \eqref{bnd_sum_h(a)_f(a)}, \eqref{error_sum_(nu-f'(a))^2}, \eqref{error_sum_(nu-f'(a)^)3}:
\begin{multline}
\label{bnd-S12-final}
|S_{12}| 
\leq  \frac{H(b)}{4\pi^2 } Z_0(\lfloor f'(a) \rfloor, b , f'(b)) + \frac{H(a)}{4\pi^2 } Z_0(\lfloor f'(a) \rfloor, a, f'(a))
\\ + \frac{H_1(a)}{4\pi^3 f'(a)} \Big[ \Big( \frac{1}{\delta^2} + \frac{1}{(\delta+1)^2} + \frac{1}{\delta+1}\Big) 
- \frac{1}{f'(a)} \Big( \log(\lfloor f'(a) \rfloor+1) - \frac{1}{\lfloor f'(a) \rfloor+1} - \frac{\Gamma'}{\Gamma}(\delta) \Big) \Big]
\\  +  \frac{H(a)|f''(a)|}{4\pi^3 f'(a)} \Big[ \frac{1}{\delta^3} + \frac{1}{(\delta+1)^3} + \frac{1}{2(\delta+1)^2} - \frac{1}{f'(a)} \Big( \frac{1}{\delta^2} + \frac{1}{\delta+1} \Big) 
\Big. \\  \Big. 
+ \frac{1}{f'(a)^2} \Big( \log(\lfloor f'(a) \rfloor+1) - \frac{1}{2(\lfloor f'(a) \rfloor+1)} - \frac{\Gamma'}{\Gamma}(\delta) \Big) \Big] , 
\end{multline}
where \begin{equation}
    H(x) = |g'(x)| + 2\pi|g(x)f'(x)|, \label{eq:defH(x)}
\end{equation} and \begin{equation}
    H_1(x)=|g''(x)|+2\pi( |g(x)f''(x)|+|g'(x)f'(x)|). \label{eq:defH_1(x)}
\end{equation}
\smallskip

In a similar manner, to establish the bound for the sum $S_2$ in \eqref{sum_S_2_Ig'_I_(gf')}, we first derive an explicit expression for the general integral $I_h(+\nu)$ involving positive frequencies. By applying integration by parts to the oscillatory factor as done for $I_h
(+\nu)$, we obtain:
\begin{multline}
    \bigg| I_h(+\nu)-\frac{1}{2\pi}\left( \frac{h(b) e^{2 \pi i (f(b) + \nu b)}}{f'(b) + \nu} - \frac{h(a) e^{2 \pi i (f(a) +\nu a)}}{f'(a) + \nu}\right) \bigg|  
     \\ \leq \frac{1}{2\pi} \bigg|  \int_a^b \frac{h'(x)}{f'(x) + \nu} e^{2\pi i (f(x) + \nu x)} dx \bigg|   + \frac{1}{2\pi} \bigg|  \int_a^b \frac{h(x)f''(x)}{(f'(x) + \nu)^2} e^{2\pi i (f(x) + \nu x)} dx \bigg| . \label{bnd1-|I_h(+nu)|} 
\end{multline}
Under the assumptions that $h$ and $|f''|$ are decreasing, and since $f'(x) + \nu$ is positive and decreasing for $\nu  \geq 1$, the quotients 
$$\frac{|h'(x)|}{(2\pi)^2 (f'(x) + \nu)^2} \quad \text{and} \quad  \frac{|h(x)f''(x)|}{(2\pi)^2 (f'(x) + \nu)^3} $$
are monotonic. Applying Lemma \ref{Lemma 2 arias}, we obtain the following bounds for the remaining integrals in \eqref{bnd1-|I_h(+nu)|}, we obtain:
\begin{equation}
    \frac{1}{2\pi} \bigg|  \int_a^b \frac{h'(x)}{f'(x) + \nu} e^{2\pi i (f(x) + \nu x)} dx \bigg|  \leq \frac{|h'(a)|}{2\pi^2 (f'(a) + \nu)^2}, \label{bnd-int-plus-h1}
\end{equation}
and
\begin{equation}
    \frac{1}{2\pi} \bigg|  \int_a^b \frac{h(x)f''(x)}{(f'(x) + \nu)^2} e^{2\pi i (f(x) + \nu x)} dx \bigg|  \leq \frac{|h(a)f''(a)|}{2\pi^2 (f'(a) + \nu)^3}. \label{bnd-int-plus-h2}
\end{equation}
Combining the bounds from \eqref{bnd-int-plus-h1} and \eqref{bnd-int-plus-h2} to \eqref{bnd1-|I_h(+nu)|}, we have:
 \begin{multline}\label{eq:S2_bound_full}
\bigg| \sum_{\nu=1}^{\infty} \frac{1}{2 \pi \nu} I_h(+\nu)\bigg| 
\leq 
\frac{|h(b)|}{4\pi^2} \bigg|  \sum_{\nu=1}^\infty \frac{e^{2\pi i \nu b}}{\nu(f'(b)+\nu)} \bigg|  + \frac{|h(a)|}{4\pi^2} \bigg|  \sum_{\nu=1}^\infty \frac{e^{2\pi i \nu a}}{\nu(f'(a)+\nu)} \bigg| \\
+ \frac{|h'(a)|}{4 \pi^3} \sum_{\nu=1}^{\infty} \frac{1}{\nu (f'(a) + \nu)^2} + \frac{|h(a) f''(a)|}{4 \pi^3}.
    \end{multline}
To estimate the boundary terms in \eqref{eq:S2_bound_full}, we invoke the bound \eqref{bnd-Splus} for $Z_1(x, y)$ from Lemma \ref{lemma_bnd_S_plus_minus}, yielding the following:
\begin{equation}
    \bigg| \frac{h(b)}{4\pi^2} \sum_{\nu=1}^\infty \frac{e^{2\pi i \nu b}}{\nu(f'(b)+\nu)} \bigg|  \leq \frac{|h(b)|}{4\pi^2 }Z_1(b,f'(b)), \label{eq:S2_bound_b}
\end{equation}
and
\begin{equation}
    \bigg| \frac{h(a)}{4\pi^2} \sum_{\nu=1}^\infty \frac{e^{2\pi i \nu a}}{\nu(f'(a)+\nu)} \bigg|  \leq \frac{|h(a)|}{4\pi^2 }Z_1(a,f'(a)). \label{eq:S2_bound_a}
\end{equation}
Next, we address the residual sums in \eqref{eq:S2_bound_full} using the bounds from \eqref{eq:sum_infty_power_k} in Lemma \ref{lem:harmonic-lead}:
\begin{equation}
\frac{|h'(a)|}{4\pi^3} \sum_{\nu=1}^\infty \frac{1}{\nu(\nu+f'(a))^2} \leq \frac{|h'(a)|}{4\pi^3} \left[ \frac{\log(f'(a)+1)+\gamma}{f'(a)^2} - \frac{1+2f'(a)}{2f'(a)^2(1+f'(a))} \right], 
\label{eq:S2_bound_+nu_power2_refined}
\end{equation}
and
\begin{equation}
\frac{|h(a)f''(a)|}{4\pi^3} \sum_{\nu=1}^\infty \frac{1}{\nu(\nu+f'(a))^3} \leq \frac{|h(a)f''(a)|}{4\pi^3} \left[ \frac{\log(f'(a)+1)+\gamma}{f'(a)^3} - \frac{3f'(a)+3f'(a)^2+1}{2f'(a)^3(1+f'(a))^2} \right]. 
\label{eq:S2_bound_+nu_power3_refined}
\end{equation}

\smallskip
This allows us to get the final estimate for $|S_2|$ as defined in \eqref{sum_S_2_Ig'_I_(gf')} from from \eqref{eq:S2_bound_b}, \eqref{eq:S2_bound_a}, \eqref{eq:S2_bound_+nu_power2_refined} and \eqref{eq:S2_bound_+nu_power3_refined} for $h=g'$ and $h=gf'$:
\begin{equation}
\label{bnd-S2-final}
\begin{split}
    |S_2| \leq
    & \frac{H(b)}{4\pi^2 }Z_1(b,f'(b))+\frac{H(a)}{4\pi^2 }Z_1(a,f'(a))
    + \frac{H_1(a)}{4\pi^3} \left[ \frac{\log(f'(a)+1)+\gamma}{f'(a)^2} - \frac{1+2f'(a)}{2f'(a)^2(1+f'(a))} \right]
    \\& + \frac{H(a)|f''(a)|}{4\pi^3} \left[ \frac{\log(f'(a)+1)+\gamma}{f'(a)^3} - \frac{3f'(a)+3f'(a)^2+1}{2f'(a)^3(1+f'(a))^2} \right], 
\end{split}
\end{equation}
where $H(x)$ and $H_1(x)$ are defined in \eqref{eq:defH(x)} and \eqref{eq:defH_1(x)}, respectively.
We conclude by substituting the bounds \eqref{bnd-S12-final} and \eqref{bnd-S2-final} for, respectively, $S_{12}$ and $S_{2}$, into \eqref{eq:defsum(atob)}:
\begin{multline}
    \bigg| \sum_{a<n \leq b} g(n) e^{2\pi i f(n)} - R_0(a,b)\bigg| \leq G(a,b)
    + \frac{1}{2\pi} \left( g(b) |\tilde{S}_1(b, \lfloor f'(a) \rfloor)| + g(a)|\tilde{S}_1(a, \lfloor f'(a) \rfloor)| \right) \\
+\frac{H(b)\mathcal{B}_{f'}(b) + H(a)\mathcal{B}_{f'}(a)}{4\pi^2}
    + \frac{H_1(a)}{4\pi^3f'(a)} \mathcal{E}_1(\delta, f'(a))+\frac{H(a)|f''(a)|}{4\pi^3f'(a)} \mathcal{E}_2(\delta, f'(a)), \label{Final_Explicit_Bound}
\end{multline}
where
\begin{align}
    \mathcal{B}_{f'}(x) &= Z_0(\lfloor f'(a) \rfloor, x, f'(x)) + Z_1(x, f'(x)), \label{def-B} \\
    \mathcal{E}_1(\delta, f'(a)) &= \frac{1}{\delta^2} + \frac{1}{(\delta+1)^2} + \frac{1}{\delta+1} \nonumber \\& - \frac{1}{f'(a)} \left( \log(\lfloor f'(a) \rfloor+1) - \frac{1}{\lfloor f'(a) \rfloor+1} - \frac{\Gamma'}{\Gamma}(\delta)  + \log(f'(a)+1)+\gamma- \frac{1+2f'(a)}{2(1+f'(a))}\right) , \label{auxilliary_error1} \\
    \mathcal{E}_2(\delta, f'(a)) &= \frac{1}{\delta^3} + \frac{1}{(\delta+1)^3} + \frac{1}{2(\delta+1)^2} - \frac{1}{f'(a)} \left( \frac{1}{\delta^2} + \frac{1}{\delta+1} \right) \nonumber \\  & + \frac{1}{f'(a)^2}\left(\log(\lfloor f'(a) \rfloor+1) - \frac{1}{2(\lfloor f'(a) \rfloor+1)} - \frac{\Gamma'}{\Gamma}(\delta)\right. \nonumber\\
    &\qquad \qquad\left.+\log(f'(a)+1)+\gamma- \frac{3f'(a)+3f'(a)^2+1}{2(1+f'(a))^2}\right), \label{auxillary_error2}
\end{align}
and $Z_0, Z_1$ are defined in Lemma \ref{lemma_bnd_S_plus_minus}.
The general result follows from replacing ${f'(a)}$ with ${f'(a)}-N$.
We conclude to \eqref{final-R0T0part1} with the alternate definition for $T_0$ as given in \eqref{def-TNab-partII}.
\end{proof}
\subsection{Proof of Corollary \ref{Explicit_B_estimate}}\label{section-proof-cor-Explicit_B_estimate}
\begin{proof}
Using \corref{cor-VDC}, we replace the $\frac{3}{\pi} \log(\beta-\alpha+2) + 4.0$ term found in \cite[Lemma 2.1]{PatelYang2024} with our explicit remainder terms.
For each integer frequency $0 \le \nu \le \lfloor f'(a) \rfloor$, let $x_\nu \in [a, b]$ satisfy $f'(x_\nu) = \nu$. We evaluate each integral in the frequency sum using stationary phase approximation \cite[Lemma 2.2]{PatelYang2024}. Summing these evaluations, we partition the remaining error into $S_1, S_2,$ and $S_3$ as in \cite[eq.~(2.32)]{PatelYang2024} to obtain:
\begin{equation}
\sum_{\nu = 1}^{\lfloor f'(a) \rfloor - 1} \int_{a}^{b} e^{2\pi i (f(x) - \nu x)} \, dx =S_1 + S_2 + S_3.
\end{equation}
Here, we have 
\begin{equation}\label{eq:bs1}
    S_1 = \sum_{\nu=1}^{\floor{f'(a)}-1} \frac{e^{2\pi i (f(x_\nu)-\nu x_\nu -1/8)}}{|f''(x_\nu)|^{1/2}},
\end{equation}
 $S_2$ is bounded as in \cite[eq.~(2.34)]{PatelYang2024}, and $S_3$ satisfies:
\begin{align} |S_3| &\le  \frac{2}{\pi} \log\!\big(f'(a) - f'(b)\big) + 1.251. \label{eq:S3-bound}
\end{align}
The constant $1.251$ in \eqref{eq:S3-bound} is obtained by explicitly evaluating the Digamma identity $-\frac{2}{\pi}\psi(1/2) = \frac{2}{\pi}(\gamma + 2\log 2) \approx 1.250009$, which refines the bound utilized in \cite[eq.~(2.35)]{PatelYang2024}. Note that our equation \eqref{eq:bs1} can be compared with \cite[eq.~(2.33)]{PatelYang2024}. Because of the way we state our Corollaries \ref{cor-VDC} and \ref{Explicit_B_estimate} with specific integer limits in the sums, we do not need to consider the extra terms in the boundary in the same way as Patel and Yang. The tradeoff for this is our parameter $\delta$ in our error terms. To obtain Patel and Yang's \cite[eq.~(2.34)]{PatelYang2024}, in our case, we use the following inequality:
\[
\floor{f'(a)}-1 \le f'(a)-f'(b) \le (b-a)h_2\lambda_2,
\]
since we are considering (for simplicity) the case $0<f'(b)<1$.
\smallskip

Following \cite[eq.(2.36)]{PatelYang2024}, extending the interior integral sum over $1 \le \nu \le \lfloor f'(a) \rfloor - 1$ to the full frequency range $0 \le \nu \le \lfloor f'(a) \rfloor$ involves adding the two boundary frequencies $\nu = 0$ and $\nu = \lfloor f'(a) \rfloor$. Applying Kershner's second-derivative estimate \cite[eq.(2.2)]{PatelYang2024} to each endpoint integral yields:
\begin{equation}\label{eq:boundary-integrals}
\sum_{\nu = 1}^{\lfloor f'(a) \rfloor - 1} \int_{a}^{b} e^{2\pi i (f(x) - \nu x)} \, dx = \sum_{\nu = 0}^{\lfloor f'(a) \rfloor} \int_{a}^{b} e^{2\pi i (f(x) - \nu x)} \, dx + \mathcal{O}^*\!\left( \frac{2.686}{\sqrt{\lambda_2}} \right).
\end{equation}
Combining \corref{cor-VDC} with the boundary integral relation \eqref{eq:boundary-integrals}, the bounds on $S_1$ and $S_2$ from \cite{PatelYang2024}, and the refined estimate \eqref{eq:S3-bound} completes the proof.
\end{proof}
\section{A 1st explicit approximate functional equation} \label{section-proof-AFE1}
The Approximate Functional Equation of the First Kind (AFE1) was established by Hardy and Littlewood in their 1921 memoir \cite{HL1921} (received in 1920). Specifically, \cite[Lemma 2]{HL1921} provides an explicit bound for the error term under the condition $|t| < 2\pi x / C$ for $C > 1$. The result states that for $\sigma \ge \sigma_0 > 0$ and $|s - 1| \ge \delta > 0$,
\begin{equation}
    \zeta(s) = \sum_{n \le x} n^{-s} - \frac{x^{1-s}}{1-s} + O(x^{-\sigma})
\end{equation}
holds uniformly in $s$ \cite[Lemma 2]{HL1921}. This fundamental result is presented as \cite[Theorem 4.11]{Titchmarsh}.

\subsection{Statement and proof of main result}
In this section, we utilize our optimized explicit van der Corput B-process (\thmref{thm-VDC}) to provide a further sharpened version of the AFE1. We establish the following theorem, which offers significantly tighter constants.
\begin{theorem} \label{thm-AFE1}
Let $s = \sigma + it$ with $\s \in (0,1]$ and $t \geq t_0>0$. If $c>\frac{1}{2\pi}$ and $(ct)\in\Z+\frac{1}{2}$, then
\begin{equation*} \bigg|  \zeta(s)
    - \sum_{1<n\leq ct}n^{-s} \bigg|  
    \leq m (c) (ct)^{-\s}  ,
\end{equation*}
with 
\begin{equation}
\label{def-m-AEF1}
m (c) = c + \frac{1}{\pi }  \left( \frac{1}{t_0} + 1 \right) \left(  \log \left(1+\frac{1}{2\pi c} \right) + \gamma - \frac{\Gamma'}{\Gamma}\left(1-\frac{1}{2\pi c}\right)  - \frac{1}{2(1+\frac{1}{2\pi c})}-\frac{1}{2}\right).
\end{equation}
For instance,
\begin{align}
& m(1) \approx 1.22773 \ \text{for} \ t_0=14.13473,    \\
&m(1) \approx 1.21268 \ \text{for} \  t_0=3\cdot10^{12}.
\end{align}
\end{theorem}
\begin{proof}
Let $x,M\in \mathbb{Z}+\frac{1}{2}$ and $1<x<M$. 
A classical Euler-Maclaurin summation formula leads to
\begin{equation} 
\zeta(s) - \sum_{1 \leq n \leq x} \frac{1}{n^s} = \sum_{x < n \leq M} \frac{1}{n^s} + s \int_M^\infty \frac{((u))}{u^{s+1}} \,du - \frac{M^{1-s}}{1-s} - \frac{1}{2}M^{-s}, \label{5.2}
\end{equation}
where $((u)) = [u] - u + 1/2$. 
Together with the bounds
    \begin{equation}\label{5.3}  \bigg| \frac{((M))}{M^s}\bigg|  \leq \frac{1/2}{M^\sigma} \ \text{and}\ 
    \bigg| s \int_M^\infty \frac{((u))}{u^{s+1}} \,du \bigg|  \leq 
    \frac{|s|}{2\sigma M^\sigma}, 
\end{equation} 
we obtain the approximation 
\begin{equation}
    \zeta(s)-\sum_{1\leq n\leq x}n^{-s}=\sum_{x <n\leq M}n^{-s}+\frac{M^{1-s}}{s-1}+\mathcal{O^*}\left(\frac{1}{2M^{\s}}+\frac{|s|}{2\s M^{\s}}\right) .\label{Euler-Zeta-EF1}
\end{equation}
We apply \thmref{thm-VDC} to the sum 
$\overline{\sum_{x<n\leq M}n^{-s}}=\sum_{x<n\leq M}n^{-\s}e^{it \log n} $ 
with $a=x, b=M, N=0, g(y)=y^{-\s}, g'(y) = -\s y^{-\s-1}$, $ f(y)=\frac{t}{2\pi}\log y$, and $f'(y)=\frac{t}{2\pi y}$. Hence 
\begin{equation} \label{sum-xtoM-Euler-Zeta-EF1}
  \sum_{x<n\leq M}n^{-s} 
  = \overline{R_0(x,M)}+\mathcal{O^*}(T_0(x,M)),
\end{equation}
where $R_0$ and $T_0$ are respectively defined in \eqref{def-RNab} and \eqref{def-TNab}. 
First, 
\begin{equation}
\label{R0-sum-xtoM-Euler-Zeta-EF1}
   \overline{R_0(x,M)} 
   = \int_{x}^M u^{-s}
\,du = \frac{x^{1-s}-M^{1-s}}{s-1} .  \end{equation} 
Putting together \eqref{Euler-Zeta-EF1} with \eqref{sum-xtoM-Euler-Zeta-EF1}, and \eqref{R0-sum-xtoM-Euler-Zeta-EF1}, the $M^{1-s}$-terms cancel out, and we get
\begin{equation}\label{bnd1-Zeta-EF1}
    \zeta(s) - 
    \sum_{1 \leq n\leq x}n^{-s}
     = \frac{x^{1-s} 
    }{s-1}
    + \mathcal{O^*}\left(\frac{1}{2M^{\s}}+\frac{|s|}{2\s M^{\s}}\right)
    +\mathcal{O^*}(T_0(x,M)) .
\end{equation}
Since, for $y \in [x,M]$, the assumption $x>\frac{t}{2\pi}$ ensures 
$ 0<f'(ct)= \frac{1}{2 \pi c }<1$, then the definition \eqref{bnd-Zxy} gives $\tilde{S_1}(u,f'(ct))=0$ for $u=ct$ or $M$. In addition, the assumption for $M$ and $ct$ to be in $\mathbb{Z}+\frac{1}{2}$ ensures $G(ct, M) =0$. This simplifies the expression \eqref{def-TNab}:
\begin{align*}
T_0(ct,M) & =  
\frac{ \s (ct)^{-\s-1} + 2\pi (ct)^{-\s} \left(\frac1{2\pi c} \right)}{2\pi^2 \left(\frac1{2\pi c}\right)}\left(\log\left(1+\frac1{2\pi c}\right) + \gamma - \frac{\Gamma'}{\Gamma}\left(1-\frac1{2\pi c} \right)-\frac{1}{2(1+\frac{1}{2\pi c})}-\frac{1}{2}\right) 
\\ & \leq \frac{(ct)^{-\s}}{\pi } \left( \frac{\s}{t_0} + 1 \right) \left( \log (1+\frac{1}{2\pi c} ) + \gamma - \frac{\Gamma'}{\Gamma}\left(1-\frac{1}{2\pi c}\right)-\frac{1}{2(1+\frac{1}{2\pi c})}-\frac{1}{2}\right). 
\end{align*}
Together with \eqref{bnd1-Zeta-EF1}, while letting $M\to\infty$, this gives 
 \[
\bigg|\zeta(s) - \sum_{1<n\leq ct}n^{-s} \bigg|
\leq  \bigg|\frac{(ct)^{1-s}}{s-1} \bigg|
 + \frac{(ct)^{-\s}}{\pi } 
\Big( \frac{\s}{t_0} + 1 \Big) \bigg( \log \Big(1+\frac{1}{2\pi c} \Big) + \gamma - \frac{\Gamma'}{\Gamma} \Big(1-\frac{1}{2\pi c}\Big) - \frac{1}{2(1+\frac{1}{2\pi c})}-\frac{1}{2}\bigg) ,
\]
which leads to the announced bound.
    \end{proof}
\subsection{Proof of \corref{cor:all_t}}\label{sec:all_t}
\begin{proof}
For $t\in \Z+\frac{1}{2}$, with $t \geq t_0$, the bound \eqref{def-m-AEF1} with $c=1$ becomes
\[
 m(1) =  1 
    + \frac{1}{\pi }  \Big( \frac{1}{t_0} + 1 \Big) \bigg(  \log \Big(1+\frac{1}{2\pi } \Big) + \gamma - \frac{\Gamma'}{\Gamma}\Big(1-\frac{1}{2\pi }\Big)-\frac{1}{2(1+\frac{1}{2\pi })}-\frac{1}{2}\bigg).   
    \]

\textit{Case 1.} 
First, consider the case where $t-\floor{t}\in[0,\frac{1}{2}]$. Then, we can write $N\le t \le N+\frac{1}{2}$, for some integer $N$ with $N\ge \floor{t_0}  \ge 14$. Let $t_2=N+\frac{1}{2} \geq t$. Define $c\in[1,14.5/14]$ by $t_2=ct$. Then, we have 
\begin{equation*} \bigg| \zeta(s)
    - \sum_{1\le n\leq t}n^{-s} \bigg| = \bigg| \zeta(s)
    - \sum_{1\le n\leq t_2}n^{-s} \bigg| 
    \leq m (c) \,  t_2^{-\s} \leq  m (c) \, t^{-\s}.
\end{equation*}
Below, we show that $m(c)$ is increasing for $c$ on the interval $[1,14.5/14]$. Therefore, if $N=\floor{t_0}$, for any $t$ such that $N\le t \le N+\frac{1}{2}$, we obtain the inequality \eqref{eq:AF1cor} with constant 
\[ c_0=m(c) =m\left(\frac{\floor{t_0}+\frac{1}{2}}{t}\right) \le m\left(\frac{\floor{t_0}+\frac{1}{2}}{t_0}\right).\]
Similarly, if $N\ge \floor{t_0}+1$, for any $t$ such that $N\le t \le N+\frac{1}{2}$, we obtain the inequality \eqref{eq:AF1cor} with constant $c_0=m\left(\frac{N+\frac{1}{2}}{N}\right)$. Since $m(c)$ is increasing, this takes its maximum value at $N=\floor{t_0}+1$.
Combining both estimates, for any $t$ such that $N\le t \le N+\frac{1}{2}$ and $N\ge \floor{t_0}\ge 14$, we obtain 
\begin{equation}\label{eq:c0-1}
  c_0=\max\left( m\left(\frac{\floor{t_0}+1/2}{t_0}\right), 
  m\left(\frac{\floor{t_0}+3/2}{\floor{t_0}+1}\right) \right).  
\end{equation} 

\textit{Case 2.} 
Now, consider the case where $t-\floor{t}\in(\frac{1}{2},1)$. Then, we can write $N+\frac{1}{2}<t<N+1$ for some fixed integer $N \geq 14$. Let $f=t-N-\frac{1}{2}$, so that $0<f<\frac{1}{2}$, and let $c = 1-\frac{f}{t}$. Note that $ct=t-f=N+\frac{1}{2}\in\Z+\frac{1}{2}$. Applying Theorem \ref{thm-AFE1} with $x=ct$, and since $\floor{ct}=\floor{t}=N$, we obtain 
\[
\bigg|  \zeta(s)
    - \sum_{1\le \leq t}n^{-s} \bigg|  = \bigg|  \zeta(s)
    - \sum_{1\le n\leq ct}n^{-s} \bigg|  
     \leq A(f) t^{-\s},
\]
where \[
A(f)=m \Big(1-\frac{f}{t}\Big)\Big(1-\frac{f}{t}\Big)^{-\sigma} \leq m \Big(1-\frac{f}{t}\Big)\Big(1-\frac{f}{t}\Big)^{-1},\] and $m \Big(1-\frac{f}{t}\Big)$ is defined in \eqref{def-m-AEF1} (we may take $t_0=14.5$). 
We claim that $A(f)$ is an increasing function of $f\in(0,\frac{1}{2})$. Indeed, consider $c=1-\frac{f}{N+\frac{1}{2}+f}$ as a decreasing function of $f$, and note that $A(f)=B(c) = \frac{m (c)}{c}$. We take derivatives to show that $B$ is a decreasing function of $c$ (and therefore $A$ is an increasing function of $f$). Note that, for $N\ge14$ we have $1 \geq c \geq 1-\frac{1}{30}$, and we restrict to this domain. 
As
\begin{align*}
    B'(c)=\frac{m '(c)\cdot c-m (c)}{c^2},
\end{align*}
it therefore suffices to show that 
\[
m '(c)\cdot c < m (c).
\]
Note that, if $y=1+\frac{1}{2\pi c}$, then $y'=-\frac{1}{2\pi c^2}<0$, and $\frac{y'}{y}= -\frac{1}{c(2\pi c+1)}$. From \eqref{def-m-AEF1}, we find
\begin{equation*}
m '(c)c = c + \frac{1}{\pi }  \left( \frac{1}{t_0} + 1 \right) \left( \frac{-1}{2\pi c+1} - \psi_1\left(1-\frac{1}{2\pi c}\right)\cdot \frac{1}{2\pi c} -\frac{\pi c}{(2\pi c+1)^2}\right). 
\end{equation*}
Here, $\psi_1(x)$ is the trigamma function - the second derivative of $\log \Gamma(x)$. 
By using interval arithmetic,\footnote{We use the MPFI library \cite{mpfi} on Sagemath 10.8.} we rigorously verify the following inequality numerically on the interval $c\in [1-\frac{1}{30},1]$:
\[
     \log \left(1+\frac{1}{2\pi c} \right) + \gamma - \frac{\Gamma'}{\Gamma}\left(1-\frac{1}{2\pi c}\right)-\frac{1}{2(1+\frac{1}{2\pi c})}-\frac{1}{2} > 
      \frac{-1}{2\pi c+1} - \psi_1\left(1-\frac{1}{2\pi c}\right)\cdot \frac{1}{2\pi c} -\frac{\pi c}{(2\pi c+1)^2}.
\]
It follows that $m '(c)\cdot c < m (c)$, and therefore $B(c)$ is decreasing with $c$ in this interval, for any fixed value of $t_0>0$. We similarly verify that $m'(c)>0$ for $c\in[1,\frac{14.5}{14}]$ with interval arithmetic, and therefore $m(c)$ is increasing in this interval.
We conclude
\[
A(f)\leq A\left(\frac{1}{2}\right)=B\left(1-\frac{1}{2N+2}\right)
\]
since $N\geq 14$.
Since $B(c)$ is decreasing in the interval $c\in [1-\frac{1}{2(N+1)},1]\subset [1-\frac{1}{30},1]$ we obtain the inequality \eqref{eq:AF1cor} with constant 
\begin{equation}
    c_0= m \left(1-\frac{1}{2(N+1)}\right)/\left(1-\frac{1}{2(N+1)}\right). 
\end{equation}
The maximum value occurs at $N=\floor{t_0}.$ We combine this with \eqref{eq:c0-1} to obtain the following constant, valid for any real $t\ge t_0\ge14$: 
\begin{equation}\label{eq:def-c0}
  c_0=\max\left( m\left(\frac{N+1/2}{t_0}\right), \, 
  m\left(\frac{N+3/2}{N+1}\right), \, 
  m \left(1-\frac{1}{2(N+1)}\right)/\left(1-\frac{1}{2(N+1)}\right)
  \right),  
\end{equation}
where $N=\floor{t_0}$ and $m$ is defined in \eqref{def-m-AEF1}. 
We conclude to the announced values for $c_0$ taking $t_0=14.13473$ and $t_0=3\cdot 10^{12}$, respectively.
\end{proof}
\section{An explicit 2nd approximate functional equation}\label{section-AFE2}
In 1921-1923, Hardy and Littlewood established the approximate functional equation in two forms. 
Lemma 15 of \cite{HL1921} states that for $0 < \sigma < 1$ and $2\pi xy = |t|$,
\begin{equation}
    \zeta(s) = \sum_{n \le x} n^{-s} + \chi(s) \sum_{n \le y} n^{s-1} + O(x^{-\sigma} \log |t|) + O(|t|^{1/2-\sigma} y^{\sigma-1} \log |t|),
\end{equation}
a version they described ``imperfect" , but ``which follows more naturally from our previous analysis and is sufficient for our immediate purpose" (namely, bounding zeta inside the critical strip 
and counting its zeros on the $1/2$-line.)
The proof truncates the Dirichlet series and its functional-equation dual, and bounds the leftover tails with exponential-sum estimates, which is where the logarithmic factor arises. 
In \cite[Theorem A]{HL1923}, they remove the $\log |t|$ factor and apply the result to divisor problems. The method starts from the contour-integral representation of zeta and evaluates it by the saddle point, so that the remainder is an asymptotic expansion in negative powers of $(t/2\pi)$, and thus without $\log t$-factor. 
Both versions are classical (see \cite[Theorem 4.13, Theorem 4.15]{Titchmarsh}). 

\smallskip
Simoni\v{c} \cite[Theorem 4]{Simonic2020} gives an explicit version of the log-free form. We give here an explicit version of \cite[Lemma 15]{HL1921}-\cite[Theorem 4.13]{Titchmarsh}, based on the van der Corput estimates of \thmref{thm-VDC}. 
Although our error term retains a factor $\log x$ (respectively $\log y$), its constants are small enough (see \corref{cor-AFE2}) that it improves on \cite[Theorem~4]{Simonic2020} in a range made explicit in \corref{cor-k-AFE2}.
\subsection{Statement of result}
\begin{theorem} \label{thm-AFE2}.
Let $s = \sigma + it$ with $ 1/2\leq \sigma \leq 1$ and $|t| \geq t_0 \geq 2\pi$. Also, assume that $x,y$ are in $\mathbb{Z}+\frac12$ and satisfy $2\pi xy = |t|$ and $x,y\geq h\geq 1.5$. 
We have the Approximate Functional Equation:
\begin{equation}\label{AFE2-explicit}
\zeta(s) = \sum_{1 \le n \le x} n^{-s} + \chi(s)\sum_{1 \le m \le y} m^{s-1} + \mathcal{E}(\s,t,x,y),
\end{equation} 
where the error term $\mathcal{E}(\s,t, x, y)$ satisfies the following bounds:
\begin{equation}\label{bnd-E-all-x-y}
\left|\mathcal{E}(\s,t, x, y)\right|
\leq 
\begin{cases}
  \Big(  \frac{\log y}{\pi}+ \mathcal{E}_0(\sigma,h,t_0) \Big) \Big(\frac{t}{2\pi}\Big)^{1/2-\sigma} y^{\sigma-1}  &\text{if}\ x \ge  y, \\
\left(\frac{C_0(\sigma,t_0)}{\pi} (\log x) +\mathcal{E}_0(\sigma, h, t_0) \right)x^{-\s} &\text{if}\ x < y.
\end{cases}
\end{equation}
with
\begin{equation}\label{def-E0-all-x-y}
    \mathcal{E}_0(\s,h,t_0) =
\begin{cases}
    A_{0}(1-\sigma,h,t_0)C_0(\sigma,t_0)  + B_{0}(1-\sigma,t_0) &\text{if}\ x \ge y,
  \\
  A_0(\s,h,t_0)  + C_0(\s, t_0) B_0(\s,t_0) \ &\text{if}\ x < y.
\end{cases}
\end{equation}
Here, $(A_0, B_0, C_0)$ are defined in \eqref{def-A0},  \eqref{def-B0}, and \eqref{def-C0} respectively. 
\end{theorem}
\begin{remark}
We note that we have the following approximations when $t_0$ is large enough, and that they are valid uniformly for values of $\sigma \in[1/2,1]$:
\[
C_0(\s,t_0)\approx 1\ \text{ and }\ B_0(\s,t_0)\approx 0,
\]
rendering $A_0(\s,h,t_0)$ (when $x > y$) and $A_0(1-\s,h,t_0) C_0(\s,t_0)$ (when $x < y$) practically identical to $\mathcal{E}_0(\s,h,t_0)$.
In addition, 
\begin{align*}
\mathcal{E}_0 \approx A_0 &\approx \frac{1}{4} + \frac{1}{\pi} \Big(\gamma + \frac{7}{2}\log 2 - \frac{3}{2}\Big) + \frac{115}{27\pi^2} 
+  \frac{7}{4\pi h}+ \frac{3}{4\pi(1+h)}  + \frac{7}{8\pi h^2}  \\
&\approx  1.160046+ \frac{0.557042}{h} + \frac{0.238732}{1+h}+\frac{0.278521}{h^2}
\\& \approx \begin{cases}
 1.750688&\ \text{if}\ h = 1.5\ (\text{cases } x\neq y),\\
1.160046 &\ \text{if}\ h \asymp \sqrt{t_0} \ (\text{case } x= y).
\end{cases}
\end{align*}
This leads to the following estimate for error term \eqref{bnd-E-all-x-y}:
\begin{equation}
\left|\mathcal{E}(\s,t, x, y)\right|
\leq 
\begin{cases}
  \left(\frac{(\log x)}{\pi}  +  1.750688 \right)x^{-\s} &\text{if}\ x > y,  \\
  \left(\frac{(\log x)}{\pi}  +1.160046 \right)x^{-\s} &\text{if}\ x = y,  \\
\Big(  \frac{\log y}{\pi}+ 1.750688  \Big) \Big(\frac{t}{2\pi}\Big)^{1/2-\sigma} y^{\sigma-1}  &\text{if}\ x<  y.
\end{cases}
\end{equation}
Thus, for $\max(\log x, \log y)$ small enough, our error term competes with previous work. 
\corref{cor-AFE2} and \corref{cor-k-AFE2} provide precise numerical insight about this. 
\end{remark}
\subsection{Proof of \thmref{thm-AFE2}}\begin{proof}
 We recall that $s = \sigma + it$, with $0 \leq \sigma \leq 1$ and $t > t_0  \geq 2\pi$. 
 Let $h \geq 1$.
 We assume $x$ and $y$ are positive numbers in $\mathbb{Z}+1/2$ satisfying the relation $2\pi xy = |t|$, along with $x, y \geq h$. 
For a sufficiently large integer $N \geq 1$, the Riemann zeta function $\zeta(s)$ can be expressed as:
\begin{equation}\zeta(s) = \sum_{1 \leq n \leq x} \frac{1}{n^s} + \sum_{x < n \leq N} \frac{1}{n^s} + \bigg( \zeta(s) - \sum_{1 \leq n \leq N} \frac{1}{n^s} \bigg) .\label{eq:zeta_split}
\end{equation}
The central sum \begin{equation}
\overline{S}=\overline{\sum_{x < n \leq N} n^{-s}}=\sum_{x < n \leq N} n^{-\s}e^{it \text{ log}n}=\sum_{x<n\leq N} n^{-\s}e^{2\pi i (\frac{t \textbf{log }n}{2 \pi})} \label{S.1}
\end{equation}
is analyzed by applying the explicit Van der Corput B established in \thmref{thm-VDC} (Part-II).
We must verify that the functions 
   \[
g(u) = u^{-\sigma} \ \text{and}\ 
f(u) = \frac{t \log u}{2\pi} 
   \]     
satisfy the necessary smoothness and monotonicity conditions on the interval $[x, N]$. 
\smallskip

For $u \in [x, N]$ and $s = \sigma + it$ with $|t| > 2\pi$ and $\sigma > 0$:
\begin{itemize}
    \item 
    We observe that $f'(u)$ is positive and strictly decreasing, while the magnitude of its second derivative, $|f''(u)| = \frac{t}{2\pi u^2}$, is likewise positive and steadily decreasing.
    \item 
    For $\sigma > 0$, $g(u)$ is a positive, decreasing function. Its first two derivatives $|g'(u)| = \sigma u^{-\sigma-1}$ and $g''(u) = \sigma(\sigma+1)u^{-\sigma-2}$ are both positive and strictly decreasing on $[x, N]$.
    \item Additionally, we note that 
    $  |g'(u)f'(u) + g(u)f''(u) 
    = \left|-\frac{t(\sigma+1)}{2\pi u^{\sigma+2}}\right|$ is  strictly positive and monotonically decreasing on $[x, N]$
\end{itemize}
Thus, the expression for 
$S$ from \eqref{S.1} become:
\begin{equation}
\begin{split}
S= \overline{R_M(x,N)}+\mathcal{O}^*(\overline{T_M(x,N)})
\label{Euler-Zeta-EF2}
    \end{split}
\end{equation}
where
$R_M(x,N)$ and $T_M(x,N)$ are defined in \eqref{def-RNab} and \eqref{def-TNab-integer+1/2-partII}.
Here, $M\leq \lfloor f'(N)\rfloor=\lfloor \frac{t}{2 \pi N} \rfloor  $ is a non-negative integer and $\delta$ is defined as  $\delta = 1- (f'(x)-\lfloor f'(x) \rfloor) = 1 - (y -\lfloor y \rfloor) = 1 - \eta$ . Take $M=0$, we have
\begin{equation}
    S = \overline{R_0(x,N)}+\mathcal{O}^*(\overline{T_0(x,N)})
\label{S}
\end{equation}
We recall that $$\overline{R_0}=\sum_{m=0}^{\lfloor f'(x) \rfloor} \int_x^N u^{-\s}e^{-{2\pi i(\frac{t \log u}{2 \pi }-m u)}} du= \sum_{m=0}^{\lfloor f'(x) \rfloor} \int_x^N u^{-s}e^{2\pi im u} du.$$ We  isolate the $m=0$ term:
\begin{equation}
\int_x^N u^{-s} \,du = \left[\frac{u^{1-s}}{1-s}\right]_x^N = \frac{N^{1-s}}{1-s} - \frac{x^{1-s}}{1-s}=\frac{N^{1-s}}{1-s}+\mathcal{O}^*\bigg(\frac{x^{1-\sigma}}{|t|}\bigg). \label{eq:int_xtoN(m=0)}
\end{equation}
Using \eqref{Euler-Zeta-EF1} from \thmref{thm-AFE1} and
rearranging to get expression for  $\frac{N^{1-s}}{1-s}$:
\begin{equation}
    \frac{N^{1-s}}{1-s}=\sum_{n\leq N}\frac{1}{n^{s}}-\zeta(s)+\mathcal{O}^* \bigg(\frac{1}{2N^{\s}}+\frac{|s|}{2\s N^{\s}}\bigg),\label{T0-sum-xtoM-Euler-Zeta-EF2}
\end{equation}
and substitute it into \eqref{eq:int_xtoN(m=0)}
\begin{equation}
    \int_x^Nu^{-s}du=\sum_{n\leq N}\frac{1}{n^{s}}-\zeta(s)+\mathcal{O}^* \bigg(\frac{1}{2N^{\s}}+\frac{|s|}{2\s N^{\s}}\bigg)+\mathcal{O}^*\bigg(\frac{x^{1-\sigma}}{|t|}\bigg).\label{eq:bnd_int_xtoN}
\end{equation}
Together with \eqref{eq:bnd_int_xtoN} 
and \eqref{S}, we get:
\begin{equation}
S= \sum_{n\leq N}\frac{1}{n^{s}}-\zeta(s)+\sum_{1\leq m \leq \lfloor\frac{t}{2 \pi x}\rfloor}\int_x^N u^{-s}e^{2 \pi i m u} du +\mathcal{O}^* \bigg(\frac{1}{2N^{\s}}+\frac{|s|}{2\s N^{\s}} \bigg)+ \mathcal{O}^* \bigg(\frac{x^{1-\s}}{|t|}\bigg)+\mathcal{O}^*(\overline{T_0(x,N)}).
\end{equation}
Substituting this sum into \eqref{eq:zeta_split}, the equation for $\zeta(s)$ takes the form:
\begin{equation}
    \zeta(s) = \sum_{1 \leq n \leq x} n^{-s} + \sum_{1\leq m \leq \lfloor\frac{t}{2 \pi x}\rfloor}\int_x^N u^{-s}e^{2 \pi i m u} du   + E_1+E_2+E_3 \label{eq:zeta-E1_E2_E3}
\end{equation}
with
 \begin{align}
 \label{def-E1}
 &       E_1=\mathcal{O}^* {\bigg(\frac{1}{2N^{\s}}+\frac{|s|}{2\s N^{\s}} \bigg)},
\\ & \label{def-E2}
E_2= \mathcal{O}^* \bigg(\frac{x^{1-\s}}{|t|}\bigg),
\\&
\label{def-E3}
E_3=\mathcal{O}^*(\overline{T_0(x,N)})=\mathcal{O}^*(T_0(x,N)),
\end{align}
and $T_0(x,N)\in\R$ is as defined in \eqref{def-TNab-integer+1/2-partII}.
Using the notation \eqref{def-intJabm}, we recognize $J(x,N,m)$ 
as the integral in \eqref{eq:zeta-E1_E2_E3}, 
and rewrite 
\[
J(x,N,m) = J(0,\infty,m) - J(0,x,m) - J(N,\infty,m).
\]
Lemma \ref{lem:u-s} took care of estimating the secondary terms $J(0,x,m)$ and $J(N,\infty,m).$
The central challenge in evaluating the integral $J(0,\infty,m)$ is its complex exponential term. We first relate it to the Gamma function (we use complex contour integration and the change of variable $t = -2\pi imu$:
\begin{equation}
J(0,\infty,m) = \int_{0}^{-i\infty} \left(\frac{-t}{2\pi i m}\right)^{-s} e^{-t} \left(\frac{-                                           1}{2\pi i m}\right) dt  
 = \bigg(\frac{2\pi m}{i}\bigg)^{s-1}\Gamma(1-s). \label{I1 1.1}
 \end{equation}
Lemma \ref{lem:gamma-chi2} then gives
 \begin{equation} \label{new I1}
   J(0,\infty,m) 
   = \chi(s) \left(1 + \mathcal{O}^* \bigg(\dfrac{ e^{-\pi t} }{1 - e^{-\pi t_0}}\bigg)\right) m^{s-1}. 
 \end{equation}
Finally, we evaluate the sum: 
\begin{equation} \label{identity sum J0infinitym}
    \sum_{1 \leq m \leq y} J(0,\infty,m) 
    = 
    \chi(s) \sum_{1 \leq m \leq y} m^{s-1} 
    + \mathcal{O}^* \bigg(|\chi(s)| \frac{e^{-\pi t}}{1 - e^{-\pi t_0}} \sum_{1 \leq m \leq y} m^{s-1} \bigg)
\end{equation}
We bound $\sum_{1 \leq m \leq y} m^{s-1}$ uniformly for $\sigma\in[0,1]$ as follows. 
Separating the term $m=1$ and using 
$m^{\sigma - 1} 
\leq \frac{y^\sigma}{m}$ for $2 \leq m \leq y$, we obtain:
\begin{equation}
\bigg|\sum_{1 \leq m \leq y} m^{s-1} \bigg| 
\leq 
1 + \sum_{2 \leq m \leq y} m^{\sigma-1} 
\leq 1 + y^\sigma \sum_{2 \leq m \leq y} \frac{1}{m} \leq y^\sigma (\log y) + 1.
\end{equation}
Using the above estimate to bound \eqref{identity sum J0infinitym}, we obtain 
 \begin{equation}\label{est I1}
  \begin{split}   \sum_{1 \leq m \leq y-\eta} J(0,\infty,m) = & \chi(s) \sum_{1 \leq m \leq y} m^{s-1} +\mathcal{O}^* \left( |\chi(s)| \,  \frac{e^{-\pi t}}{1 - e^{-\pi t_0}}\, y^{\sigma-1} \left( y\log y + y^{1-\sigma}  \right) \right).
 \end{split}\end{equation}
Together with the bounds \eqref{est I2}, \eqref{est I3}, amd \eqref{est I1} we conclude:
\begin{equation}
   \sum_{1\leq m \leq y }\int_x^N u^{-s}e^{2 \pi i m u} du=
   \chi(s) \sum_{1 \leq m \leq y} m^{s-1} +E_{4}+E_{5}+E_{6},\label{4.59}
\end{equation}
where 
\begin{align}
& \label{def-E45}
    E_{4}=\mathcal{O}^* \left( |\chi(s)| y^{\sigma-1} \left( y\log y + y^{1-\sigma}  \right) \frac{e^{-\pi t}}{1 - e^{-\pi t_0}} \right) ,
\\& \label{def-E6}
  E_5 = \mathcal{O}^*\!\left( x^{-\sigma} \left( 
\frac{1 }{ \pi }\log y + \frac{1}{\pi}\left(\gamma+2\log 2-\frac{3}{2}\right)  +\frac{3}{4\pi y} + \frac{3}{8\pi y^2} \right) \right) ,
\\& \label{def-E7}
    E_{6}=\mathcal{O}^*\bigg( 2N^{-\sigma}\bigg(\log y +1 \bigg) \bigg).
\end{align}
Therefore, \eqref{eq:zeta-E1_E2_E3} becomes
\begin{align}
    \zeta(s) = &\sum_{1 \leq n \leq x} n^{-s} + \chi(s) \sum_{1 \leq m \leq y} m^{s-1} + \sum_{1\le j \le 6} E_j, \label{eq: zeta:E1toE7}
\end{align}
where the $E_j$'s are respectively defined in \eqref{def-E1}, \eqref{def-E2}, \eqref{def-E3}, \eqref{def-E45}, \eqref{def-E6}, and \eqref{def-E7}.
Note that, taking $N\to \infty$, we immediately have 
\begin{equation}
\lim_{N \to \infty}   E_1
=0 \ \text{and}\     
\lim_{N \to \infty}E_{6}
=0.
\end{equation}
For $E_3 = \mathcal{O}^*(T_0(x,N))$ as defined in \eqref{def-TNab-integer+1/2-partII}, we get 

\begin{equation} \label{limE3}
\begin{aligned}
     \lim_{N \to \infty}E_3
     &
= \left(\frac{\log 2}{2\pi} + \frac{1}{2\pi y}\right)  x^{-\sigma} 
+ \frac{(\sigma + t)x^{-\sigma - 1}}{4\pi^2y }\left(\frac{\pi}{2}  + \frac{1}{y(\lfloor y \rfloor+1)} + \log 2+\frac{3}{2(y+1)}\right) \\ 
&  + \frac{(\sigma + t)(\sigma + 1)x^{-\sigma - 2}}{4\pi^3y}  \left(\frac{46}{9} - \frac{\log(\lfloor y \rfloor+ 1) - \frac{1}{\lfloor y \rfloor + 1} - \frac{\Gamma'}{\Gamma}(\frac{1}{2})}{y} + \frac{\log(y + 1) + \gamma}{y} - \frac{1+2y}{2y(1+y)} \right) \\ 
&  + \frac{(\sigma + t)x^{-\sigma - 3}t}{8\pi^4y} \left(\frac{230}{27} - \frac{14}{3y} + \frac{\log(\lfloor y \rfloor+ 1) - \frac{1}{2(\lfloor y \rfloor + 1)} - \frac{\Gamma'}{\Gamma}(\frac{1}{2})}{y^2} + \frac{\log(y+ 1) + \gamma}{y^2} - \frac{3y + 3y^2 + 1}{2y^2(1+y)^2} \right).
\end{aligned}
\end{equation}
We can simplify this expression significantly, as for $y  \geq 1$, the terms in the parentheses
\[ - \frac{\log(\lfloor y \rfloor+ 1) - \frac{1}{\lfloor y \rfloor + 1} - \frac{\Gamma'}{\Gamma}(\frac{1}{2})}{y}+ \frac{\log(y + 1) + \gamma}{y} - \frac{1+2y}{2y(1+y)} \]
and 
\[- \frac{14}{3y} + \frac{\log(\lfloor y \rfloor+ 1) - \frac{1}{2(\lfloor y\rfloor + 1)} - \frac{\Gamma'}{\Gamma}(\frac{1}{2})}{y^2}+\frac{\log(y+ 1) + \gamma}{y^2} - \frac{3y + 3y^2 + 1}{2y^2(1+y)^2} \]
are both strictly negative. Dropping these negative parts we replace the bound in \eqref{limE3} by:
\begin{multline}\label{bnd-E3}
 \lim_{N\to \infty} E_3 \leq  \left(\frac{\log 2}{2\pi} + \frac{1}{2\pi y}\right) x^{-\sigma} + \frac{(\sigma + t)x^{-\sigma - 1}}{4\pi^2 y}\left(\frac{\pi}{2} + \frac{1}{y(\lfloor y \rfloor+1)} + \log 2+\frac{3}{2(y+1)}\right)\\
 + \frac{46}{9} \frac{(\sigma + t)(\sigma + 1)x^{-\sigma - 2}}{4\pi^3y} + \frac{230}{27} \frac{(\sigma + t)x^{-\sigma - 3}t}{8\pi^4y} .
\end{multline}
Therefore, \eqref{eq: zeta:E1toE7} can be written as:
\begin{equation}
    \zeta(s) = \sum_{1 \leq n \leq x} n^{-s} + \chi(s) \sum_{1 \leq m \leq y} m^{s-1} +\mathcal{O}^*(\mathcal{E}(\s,t,x,y)), \ \text{for all}\ 0\le \sigma \le 1, \label{eq:AFE-2}
\end{equation}
where
\begin{equation}
    \mathcal{E}(\s,t,x,y)= \frac1{\pi}(\log y)x^{-\s}  +A(\s,t,x,y)x^{-\s}+ |\chi(\sigma+it)| B(\s,t,x,y)    y^{\s-1}. \label{epsilon_eq_A_B}
\end{equation}
Note that $E_2$, $E_3$ and $E_{5}$, as defined in \eqref{def-E2}, \eqref{bnd-E3}, and \eqref{def-E6}, contribute towards the definition of $A$, while $E_{4}$, as given in \eqref{def-E45}, defines $B$:
\begin{equation}\label{def:B(s,t,x,y)}
B(\s,t,x,y)=   \left( y\log y + y^{1-\sigma}  \right) \frac{e^{-\pi t}}{1 - e^{-\pi t_0}}.
\end{equation}
and, since $xy=\frac{t}{2\pi}$ and $\frac{x}{|t|}+\frac{1}{2\pi y} 
= \frac{1}{\pi y} $, then
\begin{multline}
A(\sigma,t,x,y) 
= \frac{1}{\pi y} 
+
\frac{1}{\pi}\left(\gamma+ 2 \log 2-\frac{3}{2}\right)  +\frac{3}{4\pi y} + \frac{3}{8\pi y^2}
 + \frac{2\log 2}{2\pi} \\+ \frac{(\sigma + t)}{2\pi t} \left(\frac{\pi}{2} + \frac{1}{y^2} + \log 2 + \frac{3}{2(y+1)}\right)  + \frac{23}{9\pi^2} \frac{(\sigma+t)(\sigma+1)}{xt} + \frac{115}{54\pi^3} \frac{\sigma+t}{x^2}.
\end{multline}
Regrouping the terms, we obtain:
\begin{multline} \label{def:A(s,t,x,y)}
A(\sigma,t,x,y) 
=
\frac{1}4 +\frac{\gamma+ \frac72\log 2-\frac{3}{2}}{\pi} 
+ \frac{115}{54\pi^3} \frac{t}{x^2}
+ \frac{7}{4\pi y} +  \frac{3}{4\pi(y+1)}
+ \frac{7}{8\pi y^2} 
\\ + \frac{23(\sigma+1)}{9\pi^2 x} 
+ \frac{115 \sigma}{54\pi^3 x^2}
 + \frac{\sigma }{4 t}
 + \frac{\sigma }{2\pi y^2 t} 
 + \frac{ \sigma  (\log 2 )}{2\pi t} 
 + \frac{ 3 \sigma }{4\pi (y+1) t}
 + \frac{23\sigma (\sigma+1)}{9\pi^2} \frac{1}{xt} .
\end{multline}
We now study the bound \eqref{epsilon_eq_A_B} depending on whether $x< y$, $x=y$, or $x>y$. 

\smallskip
\textbf{1. The Region $x > y$}: 
We introduce the parameter $h$ and assume 
\[
1.5 \leq h \leq y  \leq \sqrt{\frac{t}{2\pi}} < x = \frac{t}{2\pi y}, \  t  \geq t_0 > 0, \  \text{and} \ \frac{t}{x^2}  \leq  2\pi.
\]
In this case, $x^{-\s} = \big(\frac{t}{2\pi}\big)^{-\s} y^{\s}\leq \big(\frac{t}{2\pi}\big)^{1/2-\s} y^{\s-1}$, so that \eqref{epsilon_eq_A_B} becomes
\begin{equation}
 \mathcal{E}(\sigma,t,x,y) \leq \Big( \Big(  \frac{(\log y)}{\pi} +A(\sigma,t,x,y) \Big)\Big( \frac{t}{2\pi} \Big)^{1/2-\sigma} + |\chi(\sigma+it)| B(\sigma,t,x,y) \Big) y^{\sigma-1}. 
\end{equation}
We then apply the approximation for $|\chi(s)|$ from Lemma \ref{lem:bndchi-Simonic}, for $\frac{1}{2}\le \sigma\le 1$:
     \[
     |\chi(\sigma+it)| 
 \leq
 C_0(\sigma,t_0)  \Big( \frac{t}{2\pi} \Big)^{1/2-\sigma},
 \]
so that
\begin{equation}
   \mathcal{E}(\sigma,t,x,y) \leq \Big( 
 \frac{(\log y)}{\pi} 
   +A(\sigma,t,x,y)   + C_0(\sigma,t_0) B(\sigma,t,x,y) \Big) \Big( \frac{t}{2\pi} \Big)^{1/2-\sigma} y^{\sigma-1}.
\end{equation}
Since all terms in $A(\sigma,t,x,y)$ decrease with $y$ and $x$ where $y \geq h$ and $x\ge x_0 = \max\left(h, \sqrt{\frac{t_0}{2\pi}}\right)$, we then can bound $A$ and $B$ by the following constants (depending on the values for $h, \sigma$, and $t_0$).
First, we obtain 
\[  A(\s,t,x,y) \leq A_0(\sigma,h,t_0),\] with 
\begin{multline}\label{def-A0}
A_0(\sigma,h, t_0) 
=
\frac{1}4 +\frac{\gamma+ \frac{7}{2}\log 2-\frac{3}{2}}{\pi} 
+ \frac{115}{27\pi^2} 
+ \frac{7}{4\pi h} +  \frac{3}{4\pi(h+1)}
+ \frac{7}{8\pi h^2} 
\\+ \frac{23(\sigma+1)}{9\pi^2 x_0 } 
+ \frac{115 \sigma}{54\pi^3 x_0^2 }
 + \frac{\sigma }{4 t_0}
 + \frac{\sigma }{2\pi h^2 t_0} 
 + \frac{ \sigma  (\log 2 )}{2\pi t_0} 
 + \frac{ 3 \sigma }{4\pi (h+1) t_0}
 + \frac{23\sigma (\sigma+1)}{9\pi^2} \frac{1}{x_0  t_0} .
\end{multline}
In addition, since $y (\log y)+y^{1-\sigma} \leq \frac{1}{2} \sqrt{\frac{t}{2\pi}} \log\left(\frac{t}{2\pi}\right) + \big(\frac{t}{2\pi}\big)^{(1-\sigma)/2} $, where $1-\sigma\geq 0$ then 
\[
\big(y (\log y)+y^{1-\sigma}\big) e^{-\pi t} 
\leq \bigg(\frac{1}{2} \sqrt{\frac{t}{2\pi}} \log\Big(\frac{t}{2\pi}\Big) + \big(\frac{t}{2\pi}\big)^{(1-\sigma)/2}\bigg)  e^{- \pi t}.
\]
As the right expression decreases with $t>2\pi$, 
then $B(\sigma,t,x,y) \leq B_0(\sigma,t_0) $ with
\begin{equation}
\label{def-B0}
B_0(\sigma,t_0) = \frac{\left(\frac{1}{2} \sqrt{\frac{t_0}{2\pi}} \log\left(\frac{t_0}{2\pi}\right) + \big(\frac{t_0}{2\pi}\big)^{(1-\sigma)/2}\right)  e^{- \pi t_0}}{(1 - e^{-\pi t_0}) }.
\end{equation}
Therefore 
  \begin{equation}\label{A+B:case-x>y}
  \mathcal{E}(\sigma,t,x,y) 
  \leq \Big( \frac{(\log y)}{\pi} 
  + \mathcal{E}_0(\sigma,h,t_0) \Big) \Big(\frac{t}{2\pi}\Big)^{1/2-\sigma} y^{\sigma-1}  ,
    \end{equation}
where
\begin{equation}\label{def-E0x>y}
    \mathcal{E}_0(\s,h,t_0) = A_0(\s,h,t_0)  + C_0(\s, t_0) B_0(\s,t_0).
\end{equation}

\smallskip
\textbf{2. The Region $x = y$}:  
In this case, the bound \eqref{A+B:case-x>y}
is still valid, and has the specific shape 
 \begin{equation}\label{A+B:case-x=y}
 \mathcal{E}(\sigma,t,x,y) \leq \Big( \frac{(\log y)}{\pi} + \mathcal{E}_0\left(\sigma, h, t_0\right)\Big) y^{-\sigma} ,\end{equation}
where $h = \lfloor \sqrt{\frac{t_0}{2\pi}}\rfloor +1/2$, and $\mathcal{E}_0$ is given in \eqref{def-E0x>y}. 

\smallskip
\textbf{3. The Region $y>x$}:
Here, we introduce $h\geq 1.5$ such that:
\[
h \leq x \leq \sqrt{\frac{t}{2\pi}} < y = \frac{t}{2\pi x}, \quad t  \geq t_0 > 0.
\]
We rewrite \eqref{eq:AFE-2} for $\zeta(1-\sigma)$, swap $x$ and $y$, 
and multiply the equation by $\chi(s)$. Taking advantage of the functional equation $\zeta(s) = \chi(s)\zeta(1-s)$ and of the identity $\chi(s)\chi(1-s) = 1$, we obtain:
 \begin{equation}
    \zeta(s) = \sum_{m \leq x} m^{-s} + \chi(s)\sum_{n \leq y} n^{s-1} + \mathcal{O}^*( |\chi(s)|\mathcal{E}(1-\sigma, t, y, x)) \label{zeta:swap_x&y}
    \end{equation}
where, thanks to \eqref{epsilon_eq_A_B}, 
\begin{equation} \label{bnd-mathcalE}
\mathcal{E}(\sigma, t, x, y)
\leq |\chi(s)|\mathcal{E}(1-\sigma, t, y, x) = \left( \Big( \frac{(\log y)}{\pi}
+A(1-\sigma, t, y, x) \Big)\frac{|\chi(s)|y^{\sigma-1}}{x^{-\s}} + B(1-\sigma, t, y, x) \right)x^{-\s}.
 \end{equation}
Using Lemma \ref{lem:bndchi-Simonic} together with $ \frac{|t|}{2\pi }=xy$ with 
$\left( xy \right)^{-\sigma + 1/2}
 \frac{ y^{\sigma-1}}{x^{-\sigma}}
 = \big(\frac{x}y\big)^{1/2} <1 $, we find
\[
\frac{|\chi(s)|y^{\sigma-1}}{x^{-\sigma}} \leq
C_0(\sigma,t_0)  \left( xy \right)^{-\sigma + 1/2}
 \frac{ y^{\sigma-1}}{x^{-\sigma}}
 \leq
C_0(\sigma,t_0) .
\]
Finally, we note that the bounds $(A_0,B_0)$ for $(A,B)$ as defined in \eqref{def-A0} and \eqref{def-B0} are still valid. Consequently, for $y > x$ and $1/2\leq \sigma \leq 1$, \eqref{bnd-mathcalE} gives
\begin{equation}
\mathcal{E}(\sigma, t, y, x) \leq 
\left(
\frac{C_0(\sigma,t_0)}{\pi} (\log x)
+\mathcal{E}_0(\sigma, h, t_0) \right)x^{-\s},
\end{equation}
where here we define
\begin{equation}
\mathcal{E}_0(\sigma, h, t_0)
=A_{0}(1-\sigma,h,t_0)C_0(\sigma,t_0)  + B_{0}(1-\sigma,t_0) .
\end{equation}
\end{proof}
\begin{proof}[Proof of Corollary \ref{cor-AFE2} and \ref{cor-k-AFE2}]
For Corollary \ref{cor-AFE2}, for the stated values of $h$ and $t_0$, we find the maximum values that define $\epsilon_0$ and $\delta_0$ using Brent's method \cite{Brent}. For Corollary \ref{cor-k-AFE2}, we note that, for $h_k\le \min(x,y)\le H_k$, we have $\min(\log x,\log y)\le k$. We use the value of $h=h_k$ to compute the corresponding value of $\epsilon_0$, and the value of $\epsilon_k$ follows directly. Python code is attached to the arXiv version of this paper. 
\end{proof}
\bibliographystyle{plain}
\label{section-biblio}
\bibliography{AFE2026ago31.bib}
\ \newpage
\appendix
\section{Proof of technical lemmas} \label{secn:proofs-lemmas}
We prove here the Lemmas stated in Section \ref{secn:prel-lemmas}.
\subsection{Preliminary results about the Digamma function}\label{sec:digamma}
We recall the definition of the Digamma function $\frac{\Gamma'}{\Gamma}(z)$, where 
\[
\Gamma(z) = \int_0^{\infty} t^{z-1}e^{-t}dt\ \text{for }\Re z>0.
\]
For $x>0$, we rely on the following key inequality (see \cite{abramowitz1965handbook}):
\begin{equation}
    \log x - \frac{1}{x} < \frac{\Gamma'}{\Gamma}(x) < \log x -  \frac{1}{2x}. \label{digamma1}
\end{equation}
The Digamma function has a well-known series representation, which is valid for all $x>-1$ 
\begin{equation}
\frac{\Gamma'}{\Gamma}(1+x) = -\gamma + \sum_{n=1}^{\infty} \frac{x}{n(n+x)}
= -\gamma + \sum_{n=1}^{\infty} \left( \frac{1}n  - \frac1{(n+x)} \right) ,\label{digamma2}
\end{equation}
where $\gamma$ is the Euler-Mascheroni constant. 
\begin{equation}
\sum_{k=1}^{\infty} \left(\frac{1}{k+a} - \frac{1}{k+b}\right) 
= \frac{\Gamma'}{\Gamma}(b+1) - \frac{\Gamma'}{\Gamma}(a+1).
\label{digamma3}
\end{equation}
Finally, the function satisfies the functional equation
\begin{equation}\label{fe-digamma}
\frac{\Gamma'}{\Gamma}(x+1)=\frac1x +\frac{\Gamma'}{\Gamma}(x).
\end{equation}
 and the Digamma Duplication Formula 
 \begin{equation}
    \frac{\Gamma'}{\Gamma}(2z) = \frac{1}{2} \frac{\Gamma'}{\Gamma}(z) + \frac{1}{2} \frac{\Gamma'}{\Gamma}(z + \frac{1}{2}) + \log 2. \label{gauss digamma}
\end{equation} 
These identities are useful for evaluating various quantities arising in our proof, including to estimate the following sums. 

\subsection{Estimates of harmonic-like sums}
For $k=1,2,3$, and $N$ a positive integer, we study sums of the shape
\[
\sum_{\nu>N} \frac{1}{\nu(\nu\pm y)^k} .
\]
\begin{proof}[Proof of Lemma \ref{lem:harmonic-lead}]\ 
\begin{itemize}
\item{\it Bounding sums of the shape $\sum_{\nu>N}\frac{1}{\nu(\nu-y)^k} $ for $k=0,1,2$:}\\
Using partial fraction decomposition, 
and \eqref{digamma3}, we can rewrite
\begin{equation}\label{sum-nu(nu-y)}
\sum_{\nu>N}\frac{1}{\nu(\nu-y)} 
= \frac1{y} \sum_{\nu  \geq 1 } \left(\frac1{\nu+N-y}-\frac1{\nu+N}\right) 
= \frac1y \left( \frac{\Gamma'}{\Gamma}(N+1)-\frac{\Gamma'}{\Gamma}(N+1-y)\right).  
\end{equation}
Substituting the standard bounds \eqref{digamma1} into \eqref{sum-nu(nu-y)} and setting $\Delta = N+1-y$, we establish:
\begin{equation}
\begin{split}\label{sum-nu(nu-y)-b}
    \frac{1}{y} \Big( \log(N+1)-\frac{1}{N+1} - \frac{\Gamma'}{\Gamma}(\Delta)\Big)  &\leq \sum_{\nu>N}\frac{1}{\nu(\nu-y)} \\&
    \leq \frac{1}{y} \Big( \log(N+1)-\frac{1}{2(N+1)} - \frac{\Gamma'}{\Gamma}(\Delta)\Big),
\end{split}   
\end{equation}
which achieves proving \eqref{bnd-sum-nu(nu-y)blower-bupper}.

Before moving onto higher power of $(\nu-y)$, we note that, for $k \geq 2$, 
\[
\sum_{\nu>N} \frac1{(\nu-y)^k} 
=\frac1{(N+1-y)^k}+\frac1{(N+2-y)^k}+ \sum_{\nu>N+2} \frac1{(\nu-y)^k}.
\]
Standard integral comparison estimates directly imply:
\begin{equation}\label{sum-(nu-y)^2}
\frac{1}{\Delta^2} + \frac{1}{\Delta+1} < \sum_{\nu>N} \frac{1}{(\nu-y)^2} < \frac{1}{\Delta^2} + \frac{1}{(\Delta+1)^2} + \frac{1}{\Delta+1},
\end{equation}
and
\begin{equation}\label{sum-(nu-y)^3}
\frac{1}{\Delta^3} + \frac{1}{2(\Delta+1)^2} < \sum_{\nu>N} \frac{1}{(\nu-y)^3} < \frac{1}{\Delta^3} + \frac{1}{(\Delta+1)^3} + \frac{1}{2(\Delta+1)^2}.
\end{equation}
As partial fraction decomposition allows to rewrite
\[
\sum_{\nu>N} \frac{1}{\nu(\nu-y)^2} = - \frac1{y} \sum_{\nu>N}  \frac1{\nu(\nu -y)}+\frac1{y}\sum_{\nu>N}\frac1{(\nu-y)^2},
\]
we conclude to \eqref{bnd-sum-nu(nu-y)^2bupper} by combining \eqref{bnd-sum-nu(nu-y)blower-bupper} and \eqref{sum-(nu-y)^2}. 
Similarly, we combine 
\[
\sum_{\nu>N} \frac{1}{\nu(\nu-y)^3} = \frac1{y^2} \sum_{\nu>N}\frac1{\nu(\nu-y)} - \frac1{y^2} \sum_{\nu>N} \frac1{(\nu-y)^2} + \frac1{y}\sum_{\nu>N} \frac1{(\nu-y)^3}
\]
with \eqref{sum-nu(nu-y)-b}, \eqref{sum-(nu-y)^2} and \eqref{sum-(nu-y)^3} to conclude to \eqref{bnd-sum-nu(nu-y)^3bupper}.
\item{\it Bounding sums of the shape $\sum_{\nu=1}^{\infty}\frac{1}{\nu(\nu+y)^k} $ for $k=0,1,2$:}\\\\
For $k=0$, we proceed as similarly as above: we apply partial fraction decomposition, recognize Digamma terms using \eqref{digamma2}, and applying the dedicated bounds \eqref{digamma1}. 
We obtain
\begin{equation}\label{eq:explicit_bound_plus_y}
\begin{aligned}
\sum_{\nu=1}^\infty \frac{1}{\nu(\nu+y)} 
= \frac{1}{y} \sum_{\nu=1}^\infty \left( \frac{1}{\nu} - \frac{1}{\nu+y} \right) = \frac1{y}\Big(\frac{\Gamma'}{\Gamma}(y+1)+\gamma\Big)
\leq \frac{\log(y+1) + \gamma}{y} - \frac{1}{2y(y+1)}.
\end{aligned}
\end{equation}
In addition, for $k=2$, we utilize the decomposition and the integral lower bound $\sum_{\nu=1}^\infty \frac{1}{(\nu+y)^2} \geq \frac{1}{(1+y)}$. We obtain
\begin{equation}
\sum_{\nu=1}^\infty \frac{1}{\nu(\nu+y)^2} 
= \frac{1}{y} \sum_{\nu=1}^\infty \frac{1}{\nu(\nu+y)} - \frac{1}{y} \sum_{\nu=1}^\infty \frac{1}{(\nu+y)^2} 
\leq \frac{\log(y+1)+ \gamma}{y^2} - \frac{1+2y}{2y^2(y+1)}.\label{eq:square_residual_bound}
\end{equation}
Finally, for $k=3$, we also use the integral lower bound $\sum_{\nu=1}^\infty \frac{1}{(\nu+y)^{3}} \geq \frac{1}{2(1+y)^2}$, so that:
\begin{equation}
\begin{aligned}
\sum_{\nu=1}^\infty \frac{1}{\nu(\nu+y)^3} &= \frac{1}{y^2} \sum_{\nu=1}^\infty \frac{1}{\nu(\nu+y)} - \frac{1}{y^2} \sum_{\nu=1}^\infty \frac{1}{(\nu+y)^2} - \frac{1}{y} \sum_{\nu=1}^\infty \frac{1}{(\nu+y)^3} \\
&\leq 
\frac{\log(y+1)+\gamma}{y^3} - \frac{1+3y+3y^2}{2y^3(y+1)^2}.\label{eq:3_residual_bound}
\end{aligned}
\end{equation}
\end{itemize}
\end{proof}
\begin{lemma} \label{lemma_bnd_sum_nu_powers}
Let $N$ be a positive integer and $0<y<N+1$, and $\Delta=N+1-y$. Then
\begin{align}
        \bigg| \sum_{\nu>N}\frac{(-1)^{\nu}}{\nu(\nu-y)} \bigg| \leq &  \frac{1}{y}  \bigg| \frac{\Gamma'}{\Gamma}(\Delta) - \frac{\Gamma'}{\Gamma}\left(\frac{\Delta+1}{2}\right) - \log 2  \bigg|   +\frac{1}{y(N+1)}, \label{bound_sum_minus_nu_nu-y}
\\
        \bigg| \sum_{\nu=1}^\infty\frac{(-1)^{\nu}}{\nu(\nu+y)} \bigg| \leq & \frac{\log 2}{y} + \frac{3}{2y(y+1)} .
        \label{bound_sum_minus_nu_nu+y}
\end{align}
\end{lemma}
\begin{proof}[Proof of Lemma \ref{lemma_bnd_sum_nu_powers}]\ 
\begin{itemize}
\item {\it Bounding the alternating sum  $\sum_{\nu=1}^\infty \frac{(-1)^\nu}{\nu(\nu-y)}$:}
We apply the partial fraction identity:
\begin{equation}
\label{eq:partial_frac}
\frac{1}{\nu(\nu - y)} = \frac{1}{y} \left( \frac{1}{\nu - y} - \frac{1}{\nu} \right).
\end{equation}
Thus
\begin{equation}
\label{eq:sum_split}
\sum_{\nu > N} \frac{(-1)^\nu}{\nu(\nu - y)} = \frac{1}{y} \left[ \sum_{\nu > N} \frac{(-1)^\nu}{\nu - y} - \sum_{\nu > N} \frac{(-1)^\nu}{\nu} \right].
\end{equation}
By the Alternating Series Estimation Theorem, the second term is bounded by its first term:
\begin{equation}
\label{eq:harmonic_bound}
\bigg|  \sum_{\nu > N} \frac{(-1)^{\nu}}{\nu} \bigg|  \leq \frac{1}{N+1}.
\end{equation}
For the first term in \eqref{eq:sum_split}, let $\Delta = N+1-y$. Shifting the index by $\nu = v + N$ yields:
\begin{equation}
\label{eq:shifted_sum}
\sum_{\nu > N} \frac{(-1)^{\nu}}{\nu - y} = (-1)^N \sum_{v=1}^{\infty} \frac{(-1)^v}{v - 1 + \Delta}.
\end{equation}
Using parity in the above infinite sum allows to recognize Digamma terms: 
\begin{equation}
\sum_{v=1}^{\infty} \frac{(-1)^v}{v - 1 + \Delta} 
= \frac{1}{2} \sum_{k=1}^{\infty} \left( \frac{1}{k + \frac{\Delta-1}{2}} - \frac{1}{k + \frac{\Delta-2}{2}} \right) = \frac{1}{2} \left( \frac{\Gamma'}{\Gamma}\left( \frac{\Delta}{2} \right) - \frac{\Gamma'}{\Gamma}\left( \frac{\Delta+1}{2} \right) \right).
\end{equation}
Then, the Digamma Duplication Formula \eqref{gauss digamma} gives:
\begin{equation}
\sum_{v=1}^{\infty} \frac{(-1)^v}{v - 1 + \Delta}  
= \frac{\Gamma'}{\Gamma}(\Delta) - \frac{\Gamma'}{\Gamma}\left( \frac{\Delta+1}{2} \right) - \log 2.
\end{equation}
Substituting this into \eqref{eq:sum_split} leads to the explicit bound:
\begin{equation}
\bigg|  \sum_{\nu > N} \frac{(-1)^\nu}{\nu(\nu - y)} \bigg|  \leq \frac{1}{y} \bigg|  \frac{\Gamma'}{\Gamma}(\Delta) - \frac{\Gamma'}{\Gamma}\left( \frac{\Delta+1}{2} \right) - \log 2 \bigg|  + \frac{1}{y(N+1)}.
\end{equation}

\item{\it Bounding alternating sum  $\sum_{\nu=1}^\infty \frac{(-1)^\nu}{\nu(\nu+y)}$}  

We evaluate the sum $\sum_{\nu=1}^{\infty} \frac{(-1)^\nu}{\nu(\nu+y)}$ for $y > 0$. Using partial fraction decomposition, 
the infinite sum can be written as:
\begin{equation}
\label{eq:sum_plus_y_split}
\sum_{\nu=1}^{\infty} \frac{(-1)^\nu}{\nu(\nu+y)} = \frac{1}{y} \left[ \sum_{\nu=1}^{\infty} \frac{(-1)^\nu}{\nu} - \sum_{\nu=1}^{\infty} \frac{(-1)^\nu}{\nu+y} \right].
\end{equation}
The first term is the alternating harmonic series, which evaluates to:
\begin{equation}
\label{eq:alt_harmonic}
\sum_{\nu=1}^{\infty} \frac{(-1)^\nu}{\nu} = -\log 2.
\end{equation}
For the second sum, we regroup terms by parity and recognize Digamma terms as defined in \eqref{digamma3}:
\begin{equation}
\sum_{\nu=1}^{\infty} \frac{(-1)^\nu}{\nu+y} 
= \frac{1}{2} \sum_{k=1}^{\infty} \left( \frac{1}{k+\frac{y}{2}} - \frac{1}{k+\frac{y-1}{2}} \right)
= \frac{1}{2} \left( \frac{\Gamma'}{\Gamma}\left( \frac{y+1}{2} \right) - \frac{\Gamma'}{\Gamma}\left( \frac{y+2}{2} \right) \right).
\end{equation}
As previously, the Digamma Duplication Formula \eqref{gauss digamma} applies.  
Thus
\begin{equation}
\sum_{\nu=1}^{\infty} \frac{(-1)^\nu}{\nu+y} = \frac{\Gamma'}{\Gamma}\left( \frac{y+1}{2} \right) - \frac{\Gamma'}{\Gamma}(y+1) + \log 2. \label{alt_second_sum}
\end{equation}
Substituting results from \eqref{eq:alt_harmonic} and \eqref{alt_second_sum} back into \eqref{eq:sum_plus_y_split}, yields the final identity:
\begin{equation}
\sum_{\nu=1}^{\infty} \frac{(-1)^\nu}{\nu(\nu+y)} = \frac{1}{y} \left( \frac{\Gamma'}{\Gamma}(y+1) - \frac{\Gamma'}{\Gamma}\left( \frac{y+1}{2} \right) \right).
\end{equation}
Applying the inequalities in \eqref{digamma1}, %
and using that $\frac{\Gamma'}{\Gamma}$ is increasing, we obtain
\begin{equation}
\label{eq:log_bound_plus_y}
\bigg| \sum_{\nu=1}^{\infty} \frac{(-1)^\nu}{\nu(\nu+y)}\bigg|  \leq \frac{\log 2}{y} + \frac{3}{2y(y+1)}.
\end{equation}
For large $y$, this shows that the sum behaves asymptotically as $\frac{\log 2}{y} + \mathcal{O}(y^{-2})$.
\end{itemize}
\end{proof}
\subsection{Estimates of finite exponential sums}
We recall the definitions \eqref{def-Sxy} :
    \begin{equation}
        S_0(x,y) = \sum_{1\leq \nu\leq y} {e^{-2\pi i \nu x}}\ \text{and}\ 
        S_1(x,y) = \sum_{1\leq \nu\leq y} \frac{e^{-2\pi i \nu x}}{\nu}.
\end{equation}
\begin{proof}[Proof of Lemma \ref{lem:geometric}]
    First, note that for $0\leq y<1$, we have $S_0(x,y)=0$, and \eqref{bnd-S_0(y)} is trivially true. Now, for $x\notin \Z$ and $y \geq 1$, we consider $S_0(x,y)$ as a geometric sum, which yields
        \begin{equation*}
            S_0(x,y) = \frac{e^{-2\pi i x}\left( e^{-2\pi i x \lfloor y\rfloor} -1 \right)}
            {e^{-2\pi i x}-1}.
        \end{equation*}
        We take appropriate exponential factors in the numerator and denominator to use the identity $e^{-i\theta}-e^{i\theta}= -2i\,\sin(\theta)$ twice:
        \[
        S_0(x,y) 
        = \frac{e^{-2\pi i x}e^{-\pi i x\lfloor y \rfloor}\left( e^{-\pi i x \lfloor y\rfloor} - e^{\pi i x \lfloor y\rfloor} \right)}
            {e^{-\pi i x}(e^{-\pi i x}-e^{\pi i x})} 
        = \frac{e^{-2\pi i x}e^{-\pi i x\lfloor y \rfloor}}
            {e^{-\pi i x}} \cdot \frac{\sin (\pi x\lfloor y \rfloor)}{\sin(\pi x)}.
        \]
        Inequality \eqref{bnd-S_0(y)} follows immediately.
        \smallskip

        Now we fix a parameter $x\notin \Z$, and use this to express $S_1(x,y)$ as a Riemann-Stieltjes integral: 
        \[
            S_1(x,y) = \int_{1^-}^{y^+} \frac{1}{u} \ d_u S_0(x,u).
        \]
        Here, the notation $d_u S_0(x,u)$ is to clarify that 
        the integration is with respect to the variable $u$. Furthermore, $1^-$ and $y^+$ denote that we take the appropriate directional limits. 
        We then integrate by parts and use that $S_0(x,u)=0$ for $u<1$ to obtain
        \begin{equation*}
            S_1(x,y) = \frac{S_0(x,y)}{y} + \int_1^y \frac{S_0(x,u)}{u^2} \, du.
        \end{equation*}
        and thus
        Applying the bound \eqref{bnd-S_0(y)} for $S_0(x,y)$ yields
        \[
            |S_1(x,y)| \leq \frac{1}{|\sin(\pi x)|}\cdot\frac{1}{y} + 
            \frac{1}{|\sin(\pi x)|}\int_1^\infty \frac{du}{u^2}\\
            =  \frac{1}{|\sin(\pi x)|}\left( \frac{1}{y} +1
            \right).
        \]
        Hence, \eqref{bnd-Zxy} holds.
        \smallskip

        Finally, we consider the case $x=n+1/2$, $n\in\Z$. In this case, for $\nu\in\Z$, we have $e^{-2\pi i \nu x} =(-1)^\nu$. Therefore, 
        \begin{equation*}
            S_1(x,y) = \sum_{1\leq \nu \leq y}\frac{(-1)^\nu}{\nu} = \sum_{\nu = 1}^\infty\frac{(-1)^\nu}{\nu} - \sum_{\nu > y }\frac{(-1)^\nu}{\nu}  = \log 2 - \sum_{\nu > y }\frac{(-1)^\nu}{\nu},
        \end{equation*}
        where 
        \[
           \bigg| 
            \sum_{\nu > y }\frac{(-1)^\nu}{\nu}
            \bigg|  \leq \frac{1}{\lfloor y \rfloor +1} < \frac{1}{y}.
        \]
        Hence, \eqref{bnd-ZxySpecial} holds.
    \end{proof}
\subsection{Estimates of tails of exponential sums}
We study here
\[
\sum_{ \nu > N} \frac{e^{-2\pi i \nu x}}{\nu(\nu-y) }
\ \text{ and }\ \sum_{ \nu=1}^{\infty} \frac{e^{-2\pi i \nu x}}{\nu(\nu+y) }.
\]
\begin{proof}[Proof of Lemma \ref{lemma_bnd_S_plus_minus}]
Let $N$ be a non-negative integer, and $x,y\in\R$ satisfying $x$ is not an integer and $1\leq y< N+1$.
The case $x\in\Z+\frac{1}{2}$ is Lemma \ref{lemma_bnd_sum_nu_powers}, equations \eqref{bound_sum_minus_nu_nu-y} and \eqref{bound_sum_minus_nu_nu+y}. For the generic case, we decompose in partial fractions to obtain 
\begin{equation} \label{eq:split_tails}
\sum_{\nu>N} \frac{e^{-2\pi i \nu x}}{\nu(\nu-y)} = \frac{1}{y} \left( \sum_{\nu=N+1}^{\infty} \frac{e^{-2\pi i \nu x}}{\nu-y} - \sum_{\nu=N+1}^{\infty} \frac{e^{-2\pi i \nu x}}{\nu} \right).
\end{equation}
We estimate each tail in the right-hand side similarly to Lemma \ref{lem:geometric}. For $u\in\R$, let 
\[S(x,u) = \sum_{N+1\leq k\leq u}e^{-2\pi i k x},\]
so that by Lemma \ref{lem:geometric}, we have
\begin{equation*} 
|S(x,u)| \leq \frac{1}{|\sin(\pi x)|},
\end{equation*}
and $S(x,u)=0$ for $u<N+1$.
We integrate by parts with a Riemann-Stieltjes integral, and note that the boundary terms vanish in the respective limits: 
\begin{equation*}
    \bigg| \sum_{\nu=N+1}^{\infty} \frac{e^{-2\pi i \nu x}}{\nu-y}\bigg|  =
    \bigg| 
\int_{N+1^-}^\infty \frac{d_uS(x,u)}{u-y}
    \bigg| 
    = 
    \bigg| 
    \int_{N+1}^\infty \frac{S(x,u)}{(u-y)^2}\, du 
    \bigg| 
    \leq \frac{1}{|\sin(\pi x)|}\left(\frac{1}{N+1-y} \right).
\end{equation*}
Similarly, we estimate the second sum by
\begin{equation}\label{eq:exp/nu}
\bigg| \sum_{\nu=N+1}^{\infty} \frac{e^{-2\pi i \nu x}}{\nu}\bigg|  \le
    \frac{1}{|\sin(\pi x)|(N+1)}.
\end{equation}
Substituting these into \eqref{eq:split_tails} and applying the triangle inequality yields
\begin{equation*}
\bigg|  \sum_{\nu>N} \frac{e^{-2\pi i \nu x}}{\nu(\nu-y)} \bigg|  \leq \frac{1}{y|\sin(\pi x)|} \left( \frac{1}{N+1-y} + \frac{1}{N+1} \right).
\end{equation*}
Analogously, we use the decomposition 
\begin{equation*}
    \sum_{ \nu=1}^{\infty} \frac{e^{-2\pi i \nu x}}{\nu(\nu+y) } = \frac{1}{y} \left( \sum_{\nu=1}^{\infty} \frac{e^{-2\pi i \nu x}}{\nu}  -  \sum_{\nu=1}^{\infty} \frac{e^{-2\pi i \nu x}}{\nu+y}\right).
\end{equation*}
For the second sum in the right-hand side, we similarly obtain 
\begin{equation*}
    \bigg| \sum_{\nu=1}^{\infty} \frac{e^{-2\pi i \nu x}}{\nu+y}\bigg| 
     =
    \bigg| 
\int_{1^-}^\infty \frac{d_uS_0(x,u)}{u+y}
    \bigg| 
    = 
    \bigg| 
    \int_{1}^\infty \frac{S_0(x,u)}{(u+y)^2}\, du 
    \bigg| 
    \leq 
     \frac{1}{|\sin(\pi x)|}\left(\frac{1}{1+y} \right),
\end{equation*}
where $S_0(x,u)$ is defined in \eqref{def-Sxy}. Combining with the estimate \eqref{eq:exp/nu}, we conclude 
\begin{equation*}
    \bigg| 
     \sum_{ \nu=1}^{\infty} \frac{e^{-2\pi i \nu x}}{\nu(\nu+y) } 
    \bigg|  \leq 
    \frac{1}{y|\sin(\pi x)|} \left( 1+\frac{1}{1+y} \right).
\end{equation*}
\end{proof}
\subsection{Approximations for $\chi$ and Gamma}
We recall that the $\chi$ function is defined as in \eqref{def:chi}:
\begin{equation}
    \chi(s) = 2^s \pi^{s-1} \Gamma(1-s) \sin\left(\frac{\pi s}{2}\right). 
\end{equation} 
We can manipulate the definition \eqref{def:chi} to express the Gamma function $\Gamma(1-s)$ as in Lemma \ref{lem:gamma-chi2}.
\begin{proof}[Proof of Lemma \ref{lem:gamma-chi2}]
     Note that 
     \[
     i^{-s+2}=-e^{\frac{-i\pi s}{2}}; \qquad 2i\sin\left(\frac{\pi s}{2} \right) = i^{-s+2}(1-e^{i\pi s}).
     \]
     Inserting the latter equation into $\eqref{def:chi}$, we find 
     \[
     \chi(s) = \left( \frac{2\pi}{i}
     \right)^{s-1}\Gamma(1-s)(1-e^{i\pi s}).
     \]
     Finally, we conclude by using the identity 
     \[
     \frac{1}{1-e^{i\pi s}} = 1+ \frac{e^{i\pi s}}{1-e^{i\pi s}} 
     \]
     together with $|e^{i\pi s}|\leq e^{-\pi t}$.
\end{proof}

\end{document}